\documentclass[reqno]{amsart}
\usepackage{amssymb}
\usepackage{graphicx}

\usepackage[usenames, dvipsnames]{color}
\usepackage{verbatim}
\usepackage{mathrsfs}
\usepackage{bm}
\usepackage{cite}

\numberwithin{equation}{section}

\newtheorem{theorem}{Theorem}[section]
\newtheorem{corollary}[theorem]{Corollary}
\newtheorem{lemma}[theorem]{Lemma}
\newtheorem{prop}[theorem]{Proposition}

\theoremstyle{definition}
\newtheorem{remark}[theorem]{Remark}

\theoremstyle{definition}

\theoremstyle{definition}

\makeatletter
\def\dashint{\operatorname%
{\,\,\text{\bf-}\kern-.98em\DOTSI\intop\ilimits@\!\!}}
\makeatother

\def\\det{\text{det}}

\def\.5{\frac{1}{2}}

\newcommand{\RN}[1]{%
  \textup{\uppercase\expandafter{\romannumeral#1}}%
}

\renewcommand{\epsilon}{\varepsilon}

\newcounter{marnote}

\begin{document}
\title[Boundary Estimates for Lam\'{e} systems ]{Boundary blow-up analysis of gradient estimates for Lam\'{e} systems in the presence of $M$-convex hard inclusions}

\author[H.G. Li]{Haigang Li}
\address[H.G. Li]{School of Mathematical Sciences, Beijing Normal University, Beijing 100875, China. }
\email{hgli@bnu.edu.cn}

\author[Z.W. Zhao]{Zhiwen Zhao}

\address[Z.W. Zhao]{1. School of Mathematical Sciences, Beijing Normal University, Beijing 100875, China. }
\address{2. Bernoulli Institute for Mathematics, Computer Science and Artificial Intelligence, University of Groningen, PO Box 407, 9700 AK Groningen, The Netherlands.}
\email{zwzhao@mail.bnu.edu.cn, Corresponding author.}

\date{\today} 



\begin{abstract}
In high-contrast elastic composites, it is vitally important to investigate the stress concentration from an engineering point of view. The purpose of this paper is to show that the blowup rate of the stress depends not only on the shape of the inclusions, but also on the given boundary data, when hard inclusions are close to matrix boundary. First, when the boundary of inclusion is partially relatively parallel to that of matrix, we establish the gradient estimates for Lam\'{e} systems with partially infinite coefficients and find that they are bounded for some boundary data $\varphi$ while some $\varphi$ will increase the blow-up rate. In order to identify such novel blowup phenomenon, we further consider the general $m$-convex inclusion cases and uncover the dependence of blow-up rate on the inclusion's convexity $m$ and the boundary data's order of growth $k$ in all dimensions. In particular, the sharpness of these blow-up rates is also presented for some prescribed boundary data.
\end{abstract}

\maketitle



\section{Introduction and main results}

In high-contrast fiber-reinforced composites the densely packed fibers usually cause various of physical field concentration during use, such as the electric field and the stress field. Here we consider the stress concentration in the context of linear elasticity. The closeness of fibers complicates the problem much more. Even numerical computation of the concentrated field is difficult, since the numerical difficulties come from singularities in the derivative. So, understanding and controlling such singularity phenomenon become a significant challenge in theoretical and numerical analysis. In the past two decades, especially since Babu\u{s}ka et al's famous work \cite{BASL1999}, it has attracted a lot of interest in the engineering and mathematical literature. It is well known that  the high concentration of extreme mechanical loads or extreme electric field may appear not only in the gaps between inclusions but also in between inclusion and the matrix boundary.

In the context of electrostatics (or anti-plane elasticity), the field is the gradient of a solution to the Laplace equation and the singularities of the gradient were captured accurately. In order to investigate the important role of the closeness in such blow-up analysis, we assume the distance between two inclusions is $\varepsilon$. It is verified that when the conductivity of the inclusions degenerates to $\infty$, the blow-up rate of the gradient is $\varepsilon^{-1/2}$ in two dimensions \cite{BC1984,BLY2009,AKL2005,Y2007,Y2009}, and it is $|\varepsilon\ln\varepsilon|^{-1}$ in three dimensions \cite{BLY2009,LY2009}. While these works are related to the estimate of the blow-up rate of the gradient, there is another direction of research to characterize the singular behavior of the gradient \cite{KLY2013,ACKLY2013,KLY2014,KLY2015,LWX2019,LY2015} by introducing an explicit singularity function. For more related literature, see \cite{BLY2010,ABTV2015,BT2012,BT2013,BV2000,BCN2013,DL2019,DZ2016,K1993,LL2017,AKLLZ2009,G2015,L2012,LLBY2014,M1996}.

Even though much progress has been made for the Laplace equation of the anti-plane elasticity as mentioned above, not much is known about the gradient blow-up in the context of the full elasticity, for example, the Lam\'e system. Because the maximum principle does not hold for the system, the known methods for scalar equations can not be directly applied to Lam\'e system. Recently, the first author \cite{BLL2015,BLL2017}, collaborated with Bao and Li, obtains a pointwise upper bound of $\nabla u$ in the narrow region, by making use of an iteration technique with respect to the energy to overcome this difficulty. They prove that for $2$-convex inclusions, the optimal blow-up rate of $|\nabla u|$ is $\varepsilon^{-1/2}$ in dimension two, $(\varepsilon|\ln\varepsilon|)^{-1}$ and $\varepsilon^{-1}$ in dimensions $n=3$ and $n\geq4$, respectively, by a lower bound estimate established in \cite{Li2018}. We here would like to remark that Kang and Yu \cite{KY2019} introduce singular functions constructed by nuclei of strain, which characterize precisely the singularities of the stress in the narrow region between two hard inclusions in dimension two.

These results mentioned above can be regarded as interior estimates on the interaction between two close inclusions. In this paper, we establish the boundary gradient estimates of solutions to Lam\'e when hard inclusions approach the matrix exterior boundary. Our results give the optimal blow-up rates of the stress in all dimensions for inclusions of arbitrary shape, including the $m$-convex inclusions, the precise definition given below, and the ``relatively flat" inclusions. This is a continuation of \cite{BJL2017,HL}, where in \cite{HL} the boundedness of the interior gradient for two nearly touching inclusions with the ``flat" boundaries is derived while in \cite{BJL2017} the boundary estimates for $2$-convexity inclusions close to boundary are studied.

Besides the free boundary value feature of this problem, due to the interaction from given boundary data on the exterior boundary, solutions will become more irregular near the boundary. In fact, even the boundary and boundary value are both smooth, it may cause a definite increase of the blow-up rate of the gradient comparing with the known interior estimates \cite{HL}. The objectives of this paper are to investigate the influence on the blow-up rate from two kinds of parameters: $(a)$ $m$, the order of domains convexity, namely, the order of the relative convexity between inclusions and matrix's exterior boundary; $(b)$ $k$, the order of growth of boundary data $\varphi$. These estimates allow us to understand completely the underlying mechanism of such stress concentration phenomenon in high-contrast composite materials.

To formulate our main results precisely, we first describe our domain and notations. Let $D\subset\mathbb{R}^{n}\,(n\geq2)$ be a bounded open set with $C^{2,\alpha}~(0<\alpha<1)$ boundary, having a subdomain $D_{1}^{\ast}\subseteq D$ with $C^{2,\alpha}$ boundary as well. We assume that $D$ and $D_{1}^{\ast}$ have a part of common boundary $\Sigma'$, that is, by a translation and rotation of the coordinates, if necessary,
\begin{align*}
\partial D_{1}^{\ast}\cap\partial D=\Sigma'\subset\mathbb{R}^{n-1}.
\end{align*}
Here $\Sigma'$ is a bounded convex domain with the origin as its mass center or $\Sigma'=\{0'\}$. Throughout the paper, we use superscript prime to denote ($n-1$)-dimensional domains and variables, such as $\Sigma'$ and $x'$. After a translation, we set
\begin{align*}
D_{1}^{\varepsilon}:=D_{1}^{\ast}+(0',\varepsilon).
\end{align*}
For the sake of simplicity, denote
\begin{align*}
D_{1}:=D_{1}^{\varepsilon},\quad\mathrm{and}\quad\Omega:=D\setminus\overline{D}_{1}.
\end{align*}

Assume that $\Omega$ and $D_{1}$ are occupied, respectively, by two different isotropic and homogeneous materials with different Lam$\mathrm{\acute{e}}$ constants $(\lambda,\mu)$ and $(\lambda_{1},\mu_{1})$. Then the elasticity tensors for the background and the inclusion, $\mathbb{C}^0$ and $\mathbb{C}^1$, respectively are
$$C_{ijkl}^0=\lambda\delta_{ij}\delta_{kl} +\mu(\delta_{ik}\delta_{jl}+\delta_{il}\delta_{jk}),$$
and
$$C_{ijkl}^1=\lambda_1\delta_{ij}\delta_{kl} +\mu_1(\delta_{ik}\delta_{jl}+\delta_{il}\delta_{jk}),$$
where $i, j, k, l=1,2,\cdots,n$ and $\delta_{ij}$ is the kronecker symbol: $\delta_{ij}=0$ for $i\neq j$, $\delta_{ij}=1$ for $i=j$.

Let $u=(u^{1},u^{2},\cdots,u^{n})^{T}:D\rightarrow\mathbb{R}^{n}$ be the displacement field, $\chi_{\Omega}$ be the characteristic function of $\Omega\subset\mathbb{R}^{n}$. Given a vector-valued function $\varphi=(\varphi^{1},\varphi^{2},\cdots,\varphi^{n})^{T}$, we consider the Dirichlet problem for the Lam$\mathrm{\acute{e}}$ system
\begin{align}\label{La.001}
\begin{cases}
\nabla\cdot \left((\chi_{\Omega}\mathbb{C}^0+\chi_{D_{1}}\mathbb{C}^1)e(u)\right)=0,&\hbox{in}~~D,\\
u=\varphi, &\hbox{on}~~\partial{D},
\end{cases}
\end{align}
where
$$e(u)=\frac{1}{2}\left(\nabla u+(\nabla u)^{T}\right)$$
is the strain tensor.
Under the assumption of the standard ellipticity condition for (\ref{La.001}), that is,
\begin{align*}
\mu>0,\quad n\lambda+2\mu>0,\quad \mu_1>0,\quad n\lambda_1+2\mu_1>0,
\end{align*}
for a given $\varphi\in H^{1}(D;\mathbb{R}^{n})$, it is well know that there is a unique solution $u\in H^{1}(D;\mathbb{R}^{n})$ for the Dirichlet problem (\ref{La.001}), which minimizes the following energy functional
$$J_1[u]:=\frac{1}{2}\int_\Omega \left((\chi_{\Omega}\mathbb{C}^0+\chi_{D_{1}}\mathbb{C}^1)e(u), e(u)\right)dx$$
in the admissible function space $H^1_\varphi(D; \mathbb{R}^{n}):=\left\{u\in  H^1(D; \mathbb{R}^{n})~\big|~ u-\varphi\in  H^1_0(D; \mathbb{R}^{n})\right\}.$

Let
$$\Psi:=\{\psi\in C^1(\mathbb{R}^{n}; \mathbb{R}^{n})\ |\ \nabla\psi+(\nabla\psi)^T=0\}$$
be the linear space of rigid displacement in $\mathbb{R}^{n}$. We know that
$$\left\{\;e_{i},\;x_{k}e_{j}-x_{j}e_{k}\;\big|\;1\leq\,i\leq\,n,\;1\leq\,j<k\leq\,n\;\right\}$$
forms a basis of $\Psi$, where $\{e_{1},\cdots,e_{n}\}$ is the standard basis of $\mathbb{R}^{n}$. We number this basis as $\left\{\psi_{\alpha}\big|\,\alpha=1,2,\cdots,\frac{n(n+1)}{2}\right\}$.

For fixed $\lambda$ and $\mu$, we denote by $u_{\lambda_{1},\mu_{1}}$ the solution of (\ref{La.001}). Similar to the appendix in \cite{BLL2015}, it follows that
\begin{align*}
u_{\lambda_1,\mu_1}\rightarrow u\quad\hbox{in}\ H^1(D; \mathbb{R}^{n}),\quad \hbox{as}\ \min\{\mu_1, n\lambda_1+2\mu_1\}\rightarrow\infty,
\end{align*}
where $u\in H^1(D; \mathbb{R}^{n})$ verifies the following free boundary value problem
\begin{align}\label{La.002}
\begin{cases}
\mathcal{L}_{\lambda, \mu}u:=\nabla\cdot(\mathbb{C}^0e(u))=0,\quad&\hbox{in}\ \Omega,\\
u=C^{\alpha}\psi_{\alpha},&\hbox{on}\ \partial{D}_{1},\\
\int_{\partial{D}_{1}}\frac{\partial u}{\partial \nu_0}\Big|_{+}\cdot\psi_{\alpha}=0,&\alpha=1,2,\cdots,\frac{n(n+1)}{2},\\
u=\varphi,&\hbox{on}\ \partial{D},
\end{cases}
\end{align}
where the free constants $C^{\alpha}$ are determined later by the third line and
\begin{align*}
\frac{\partial u}{\partial \nu_0}\Big|_{+}&:=(\mathbb{C}^0e(u))\nu=\lambda(\nabla\cdot u)\nu+\mu(\nabla u+(\nabla u)^T)\nu,
\end{align*}
and $\nu$ is the unit outer normal of $D_{1}$. Here and throughout this paper the subscript $\pm$ shows the limit from outside and inside the domain, respectively. There has established the existence, uniqueness and regularity of weak solutions to (\ref{La.002}) in \cite{BLL2015}. The $H^{1}$ weak solution of (\ref{La.002}) is proved in $C^1(\overline{\Omega};\mathbb{R}^{n})\cap C^1(\overline{D}_{1};\mathbb{R}^{n})$ and $u$ is also the unique function minimizing the energy functional as follows:
$$I_\infty[u]:=\min_{v\in\mathcal{A}}\frac{1}{2}\int_{\Omega}(\mathbb{C}^0e(v), e(v))\,dx,$$
where
\begin{equation*}\label{def_A}
\mathcal{A}:=\left\{v\in H^1_\varphi(D;\mathbb{R}^{n}) ~\Big|~ e(v)=0\ \ \hbox{in}~~D_{1}\right\}.
\end{equation*}
Now we assume that there exists a small constant $R>0$, independent of $\varepsilon$, such that $\Sigma'\subset B_{R}'$ and the top and bottom boundaries corresponding to the narrow region between $D_{1}$ and $D$ can be described as follows:
\begin{align*}
x_{n}=\varepsilon+h_{1}(x')\quad\mathrm{and}\quad x_{n}=h(x'),\quad\quad x'\in\;B_{2R}'.
\end{align*}
Let $d(x'):=\mathrm{dist}(x',\Sigma')$, and $h_{1}$ and $h$ satisfy
\begin{enumerate}
{\it\item[(\bf{\em H1})]
$h_{1}(x')=h(x')\equiv0,\;\mbox{if}\;\,x'\in\Sigma',$
\item[(\bf{\em H2})]
$\kappa_{1}d^{m}(x')\leq h_{1}(x')-h(x')\leq\kappa_{2}d^{m}(x'),\;\mbox{if}\;\,x'\in B'_{2R}\setminus\overline{\Sigma'},$
\item[(\bf{\em H3})]
$|\nabla_{x'}h_{1}(x')|,\,|\nabla_{x'}h(x')|\leq \kappa_{3}d^{m-1}(x'),\;\mbox{if}\;\,x'\in B_{2R}',$
\item[(\bf{\em H4})]
$\|h_{1}\|_{C^{2,\alpha}(B'_{2R})}+\|h\|_{C^{2,\alpha}(B'_{2R})}\leq \kappa_{4},$}
\end{enumerate}
where $\kappa_{i},i=1,2,3,4$, are four positive constants independent of $\varepsilon$. When $\Sigma'$ is a bounded convex region, we consider $m>2$; when $\Sigma'=\{0'\}$, we consider $m\geq2$.

For $z'\in B'_{R},\,0<t\leq2R$, let
\begin{align*}
\Omega_{t}(z'):=&\left\{x\in \mathbb{R}^{n}~\big|~h(x')<x_{n}<\varepsilon+h_{1}(x'),~|x'-z'|<{t}\right\}.
\end{align*}
We will use the abbreviated notation $\Omega_{t}$ for the domain $\Omega_{t}(0')$. For the simplicity of notations, for $i=0,2$, and $i=k,k+1$, $k$ is a positive integer, we denote
\begin{align*}
\rho_{i}(n,m;\varepsilon)=&
\begin{cases}
\varepsilon^{\frac{n+i-1}{m}-1},&m>n+i-1,\\
|\ln\varepsilon|,&m=n+i-1,\\
1,&m<n+i-1.
\end{cases}
\end{align*}

Before stating the main results, we introduce a family of linear functionals of $\varphi$,
$$Q_{\beta}[\varphi]:=\int_{\partial D_{1}}\frac{\partial u_{0}}{\partial\nu_{0}}\Big|_{+}\cdot\psi_{\beta},\quad\quad\beta=1,2,\cdots,\frac{n(n+1)}{2},$$
where $u_{0}\in C^{1}(\overline{\Omega};\mathbb{R}^{n})\cap C^{2}(\Omega;\mathbb{R}^{n})$ solves the following problem:
\begin{align}\label{Lek2.012}
\begin{cases}
\mathcal{L}_{\lambda,\mu}u_{0}=0,\quad\quad\;\,&\mathrm{in}\;\,\Omega,\\
u_{0}=0,\quad\quad\;\,&\mathrm{on}\;\,\partial D_{1},\\
u_{0}=\varphi(x),\quad\;\,&\mathrm{on}\;\,\partial D,
\end{cases}
\end{align}
where $\varphi\in C^{2}(\partial D;\mathbb{R}^{n})$. Let
\begin{align}\label{NCE1}
Q_{\beta;\mathrm{I}}[\varphi]=\max\limits_{1\leq\beta\leq n}|Q_{\beta}[\varphi]|,\quad Q_{\beta;\mathrm{II}}[\varphi]=\max\limits_{n+1\leq\beta\leq\frac{n(n+1)}{2}}|Q_{\beta}[\varphi]|.
\end{align}
In addition, assume that for some $\delta_{0}>0$,
\begin{align*}
\delta_{0}\leq\mu,n\lambda+2\mu\leq\frac{1}{\delta_{0}}.
\end{align*}
Unless otherwise stated, in what following $C$ denotes a constant, whose values may vary from line to line, depending only on $\kappa_{1},\kappa_{2},\kappa_{3},\kappa_{4},\delta_{0},R$ and an upper bound of the $C^{2,\alpha}$ norms of $\partial D_{1}$ and $\partial D$, but not on $\varepsilon$. We also call a constant having such dependence a $universal$ $constant$. Without loss of generality, we let $\varphi(0)=0$. Otherwise, we replace $u$ by $u-\varphi(0)$ throughout this paper. For simplicity of discussion, we assume that convexity index $m$ and growth order index $k$ are all positive integers in the following.

\begin{theorem}\label{Lthm066}
Assume that $D_{1}\subset D\subseteq\mathbb{R}^{n}\,(n\geq2)$ are defined as above, conditions ({\bf{\em H1}})--({\bf{\em H4}}) hold with $\Sigma'=B'_{r}(0')$ for some $0<r<R$, and $\varphi\in C^{2}(\partial D;\mathbb{R}^{n})$. Let $u\in H^{1}(D;\mathbb{R}^{n})\cap C^{1}(\overline{\Omega},\mathbb{R}^{n})$ be the solution of (\ref{La.002}). Then for a sufficiently small $\varepsilon>0$, for $x\in\Omega_{R}$,
\begin{align}\label{LGT316}
|\nabla u(x)|\leq C\Bigg(&\frac{Q_{\beta;\mathrm{I}}[\varphi]}{\varepsilon+d^{m}(x')}\frac{\varepsilon}{|\Sigma'|+\varepsilon\rho_{0}(n,m;\varepsilon)}+\frac{|\varphi(x',h(x'))|}{\varepsilon+d^{m}(x')}\notag\\
&+\frac{Q_{\beta;\mathrm{II}}[\varphi]}{\varepsilon+d^{m}(x')}\frac{(\varepsilon+|x'|)\varepsilon}{|\Sigma'|^{\frac{n+1}{n-1}}+\varepsilon\rho_{2}(n,m;\varepsilon)}+\|\varphi\|_{C^{2}(\partial D)}\Bigg),
\end{align}
and
\begin{align*}
\|\nabla u\|_{L^{\infty}(\Omega\setminus\Omega_{R})}\leq C\|\varphi\|_{C^{2}(\partial D)},
\end{align*}
where $d(x'):=\mathrm{dist}(x',\Sigma')$.

\end{theorem}

\begin{remark}
If $\varphi=C$ on $\partial D$, then $u\equiv C$. So Theorem \ref{Lthm066} is trivial. Thus, we only consider the case that $\varphi\not\equiv C$ in the following. Denote the bottom boundary of $\Omega_{r}$ by
\begin{align*}
\Gamma^{-}_{r}=\left\{x\in\mathbb{R}^{n}\big|\,x_{n}=h(x'),\;|x'|<r\right\},\quad\;\,0<r\leq2R.
\end{align*}
In Theorem \ref{Lthm066}, if we assume additionally
\begin{align*}
|\varphi(x)|\leq \eta\, d^{k}(x'),\quad\;\,\mathrm{on}\;\Gamma^{-}_{R},\;\mathrm{for}\;k>2,\eta\geq0,
\end{align*}
then by using the proof of Proposition \ref{lekm323} below with minor modification, in particular, we can obtain that if $m<k+1$,
\begin{align*}
|Q_{\beta}[\varphi]|\leq&C\Big(\eta(1+|\Sigma'|^{\frac{n-2}{n-1}})+\|\varphi\|_{C^{2}(\partial D)}\Big),\quad\;\,\mathrm{for}\;\beta=1,2,\cdots,n,
\end{align*}
and
\begin{align*}
|Q_{\beta}[\varphi]|\leq&C\Big(\eta(1+|\Sigma'|)+\|\varphi\|_{C^{2}(\partial D)}\Big),\quad\;\,\mathrm{for}\;\beta=n+1,\cdots,\frac{n(n+1)}{2}.
\end{align*}
Consequently, the upper bound in (\ref{LGT316}) is bounded provided $|\Sigma'|>0$,
\begin{align*}
\|\nabla u\|_{L^{\infty}(\Omega)}\leq&C\left(\frac{\eta(1+|\Sigma'|^{\frac{n-2}{n-1}})+\|\varphi\|_{C^{2}(\partial D)}}{|\Sigma'|+\varepsilon\rho_{0}(n,m;\varepsilon)}+\frac{\eta(1+|\Sigma'|)+\|\varphi\|_{C^{2}(\partial D)}}{|\Sigma'|^{\frac{n+1}{n-1}}+\varepsilon\rho_{2}(n,m;\varepsilon)}\right),
\end{align*}
which means that there is no blow-up for $|\nabla u|$.

\end{remark}
\begin{remark}
We would like to point out that here $Q_{\beta}[\varphi]$ play more roles than in \cite{HL} for interior estimates. As shown in \cite{HL}, there $Q_{\beta}[\varphi]$ is bounded and it is only a blow-up factor to determine whether blow-up occurs or not. While, in our boundary estimate case, $Q_{\beta}[\varphi]$ will increase singularity for some $\varphi$ and make the blowup rate larger than before. We can see this from the following simple example. For $i=1,2,\cdots,n$, let
\begin{align*}
\varphi^{i}(x)=&\eta|x'|^{k}\quad\;\,\mathrm{on}\;\Gamma^{-}_{R},\;\mathrm{for}\;k\geq2,\,\eta>0,
\end{align*}
and
\begin{align*}
h_{1}(x')-h(x')=&d^{m}(x')\quad\;\,\mathrm{in}\;B'_{R}.
\end{align*}
Then, it follows from the proof of Lemma \ref{KM323} below with minor modification that for $\beta=1,2,\cdots,n$,
\begin{align*}
\frac{\eta(|\Sigma'|^{\frac{n+k-1}{n-1}}+\varepsilon\rho_{k}(n,m;\varepsilon))}{C\varepsilon}\leq|Q_{\beta}[\varphi]|\leq&\frac{C\eta(|\Sigma'|^{\frac{n+k-1}{n-1}}+\varepsilon\rho_{k}(n,m;\varepsilon))}{\varepsilon}.
\end{align*}
In the following we explore the singularities of $Q_{\beta}[\varphi]$ arising from different classes of the boundary data $\varphi$. For simplicity, we consider the case that $\Sigma'=\{0'\}$.
\end{remark}

In order to classify the effect of $\varphi=(\varphi^{1},\varphi^{2},\cdots,\varphi^{n})$ on the singularities of the stress, we further assume that
\begin{itemize}
{\it
\item[(\bf{\em H5})] $h_{1}(x')-h(x')$ is even with respect to each $x_{j}$, $j=1,\cdots,n-1$, for $|x'|<R$.}
\end{itemize}
We now classify the given boundary data $\varphi$ according to its parity as follows. Unless otherwise stated, in the following we let $\varphi^{i}(x)\neq0$ on $\Gamma^{-}_{R}$, $i=1,2,\cdots,n$. Assume that for $x\in\Gamma^{-}_{R}$,
\begin{itemize}
{\it
\item[(\bf{\em A1})] for $i=1,2,\cdots,n$, $j=1,\cdots,n-1$, $\varphi^{i}(x)$ is an even function of each $x_{j}$;
\item[(\bf{\em A2})] if $n=2$, for $i=1,2$, $\varphi^{i}(x)$ is odd with respect to $x_{1}$; if $n\geq3$, for $i=1,\cdots,n-1$, $\varphi^{i}(x)$ is odd with respect to some $x_{j_{i}}$, $j_{i}\in\{1,\cdots,n-1\}$, and $\varphi^{n}(x)$ is odd with respect to $x_{1}$ and even with respect to each $x_{j}$, $j=2,\cdots,n-1$;
\item[(\bf{\em A3})] if $n=2$, $\varphi^{1}(x)$ is odd with respect to $x_{1}$, and $\varphi^{2}(x)=0$; if $n\geq3$, for $i=1,\cdots,n-1$, $\varphi^{i}(x)$ is odd with respect to $x_{i}$, and $\varphi^{n}(x)$ is odd with respect to $x_{1}$ and $x_{2}$, respectively.}
\end{itemize}

Then, we have
\begin{prop}\label{lekm323}
Assume that $D_{1}\subset D\subseteq\mathbb{R}^{n}\,(n\geq2)$ are defined as above and conditions ({\bf{\em H1}})--({\bf{\em H5}}) hold with $\Sigma'=\{0'\}$. Let $u_{0}\in C^{1}(\overline{\Omega};\mathbb{R}^{n})\cap C^{2}(\Omega;\mathbb{R}^{n})$ is the solution to (\ref{Lek2.012}). Assume that ({\bf{\em A1}}), ({\bf{\em A2}}) or ({\bf{\em A3}}) holds. If $\varphi\in C^{2}(\partial D;\mathbb{R}^{n})$ satisfies the $k$-order growth condition,
\begin{align}
|\varphi(x)|\leq \eta\,|x|^{k},\quad\;\,\mathrm{on}\;\Gamma^{-}_{R},\label{zw11}
\end{align}
for some integer $k>0$ and a positive constant $\eta$. Then, for a sufficiently small $\varepsilon>0$, we obtain that
\begin{align}\label{HND1112}
|Q_{\beta}[\varphi]|\leq&C(\eta\rho_{0}(n,m;\varepsilon)\rho_{A}(\varepsilon)+\|\varphi\|_{C^{2}(\partial D)}),\quad\;\,\beta=1,2,\cdots,n,
\end{align}
and
\begin{align}\label{HND1116}
|Q_{\beta}[\varphi]|\leq&C(\eta\rho_{2}(n,m;\varepsilon)\rho_{B}(\varepsilon)+\|\varphi\|_{C^{2}(\partial D)}),\quad\;\,\beta=n+1,\cdots,\frac{n(n+1)}{2},
\end{align}
where
\begin{align}
\rho_{A}(\varepsilon)=&
\begin{cases}
\rho_{k}(n,m;\varepsilon)/\rho_{0}(n,m;\varepsilon),&\text{for\;case\;({\bf{\em A1}})},\\
1/\rho_{0}(n,m;\varepsilon),&\text{otherwise},
\end{cases}\label{MNM86}\\
\rho_{B}(\varepsilon)=&\begin{cases}
\rho_{k+1}(n,m;\varepsilon)/\rho_{2}(n,m;\varepsilon),&\text{for\;case\;({\bf{\em A2}})},\\
1/\rho_{2}(n,m;\varepsilon),&\text{otherwise}.
\end{cases}\label{MNM87}
\end{align}

\end{prop}

\begin{remark}
If we only assume condition (\ref{zw11}) holds with respect to $\varphi$, we can obtain
\begin{align*}
|Q_{\beta}[\varphi]|\leq&C(\eta\rho_{k}(n,m;\varepsilon)+\|\varphi\|_{C^{2}(\partial D)}),\quad\;\,\beta=1,2,\cdots,n,
\end{align*}
and
\begin{align*}
|Q_{\beta}[\varphi]|\leq&C(\eta\rho_{k+1}(n,m;\varepsilon)+\|\varphi\|_{C^{2}(\partial D)}),\quad\;\,\beta=n+1,\cdots,\frac{n(n+1)}{2}.
\end{align*}
\end{remark}

In light of decomposition (\ref{Le2.015}) for $|\nabla u|$, it follows from Proposition \ref{lekm323}, Corollary \ref{co.001}, Lemma \ref{lem89} and Proposition \ref{proposition123} below that
\begin{corollary}\label{CORO987}
Assume as in Proposition \ref{lekm323}. Let $u\in H^{1}(D;\mathbb{R}^{n})\cap C^{1}(\overline{\Omega},\mathbb{R}^{n})$ be the solution of (\ref{La.002}). Then for a sufficiently small $\varepsilon>0$, we obtain that
\begin{align}\label{DALN666}
|\nabla u(x)|\leq&\frac{C}{\varepsilon+|x'|^{m}}\left[\eta\rho_{A}(\varepsilon)+\frac{\|\varphi\|_{C^{2}(\partial D)}}{\rho_{0}(n,m;\varepsilon)}+|x'|\left(\eta\rho_{B}(\varepsilon)+\frac{\|\varphi\|_{C^{2}(\partial D)}}{\rho_{2}(n,m;\varepsilon)}\right)\right]\notag\\
&+\frac{\eta|x'|^{k}}{\varepsilon+|x'|^{m}}+C\|\varphi\|_{C^{2}(\partial D)},\quad\;\,x\in\Omega_{R}.
\end{align}
where $\rho_{A}(\varepsilon)$ and $\rho_{B}(\varepsilon)$ are defined by (\ref{MNM86}) and (\ref{MNM87}), respectively.
\end{corollary}

\begin{remark}
Under assumption ({\bf{\em A1}}), the upper bound in (\ref{DALN666}) implies that

$(1)$ if $m<n$ for any $k>0$ or if $m=n$ for $k=1$, then it achieves its maximum only at the shortest line $\{x'=0'\}\cap\Omega$;

$(2)$ if $m=n$ for $k>1$ or if $m\geq n+1$ for $k<m-n+1$, then the maximum of the upper bound achieves at $\{x'=0'\}\cap\Omega$ and $\{|x'|=\varepsilon^{\frac{1}{m}}\}\cap\Omega$ simultaneously;

$(3)$ if $m\geq n+1$ for $k\geq m-n+1$, then the maximum attains at $\{|x'|=\varepsilon^{\frac{1}{m}}\}\cap\Omega$.
\end{remark}

\begin{remark}
Under assumption ({\bf{\em A2}}) or ({\bf{\em A3}}), the upper bound in (\ref{DALN666}) shows that for any $k>0$,

$(1)$ if $m<n$, then its maximum achieves only at $\{x'=0'\}\cap\Omega$;

$(2)$ if $m=n$, then the maximum achieves at both $\{x'=0'\}\cap\Omega$ and $\{|x'|=\varepsilon^{\frac{1}{m}}\}\cap\Omega$;

$(3)$ if $m>n$, then the maximum attains at $\{|x'|=\varepsilon^{\frac{1}{m}}\}\cap\Omega$.
\end{remark}

\begin{remark}
When $k\geq m-n$ and $m>n+1$, (\ref{DALN666}) becomes
\begin{align}\label{jklnm189}
|\nabla u(x)|\leq&C\left(\frac{\eta+\|\varphi\|_{C^{2}(\partial D)}}{\varepsilon^{n/m}}+\frac{|x'|^{k}}{\varepsilon+|x'|^{m}}\right),\quad\;\,\mathrm{for}\; x\in\Omega_{R}.
\end{align}
Letting $m,k\rightarrow\infty$ simultaneously,  condition (\ref{zw11}) implies $\varphi(x)\equiv\varphi(0)$ on $\Gamma^{-}_{R}$ and estimate (\ref{jklnm189}) becomes $|\nabla u(x)|\leq C$ in $\Omega_{R}$.
\end{remark}

Finally, we establish the lower bounds of $|\nabla u|$ to show that the blow-up rates obtained in Corollary \ref{CORO987} are optimal, for some explicit $\varphi$. Denote $\Omega^{\ast}:=D\setminus\overline{D^{\ast}_{1}}$. Similarly as before, we introduce a family of linear functionals with respect to $\varphi$,
$$Q^{\ast}_{\beta}[\varphi]:=\int_{\partial D_{1}}\frac{\partial u_{0}^{\ast}}{\partial\nu_{0}}\Big|_{+}\cdot\psi_{\beta},\quad\quad\beta=1,2,\cdots,\frac{n(n+1)}{2},$$
where $u_{0}^{\ast}$ is a solution of the following problem:
\begin{align}\label{LLee.008}
\begin{cases}
\mathcal{L}_{\lambda,\mu}u_{0}^{\ast}=0,\quad\quad\;\,&\mathrm{in}\;\,\Omega,\\
u_{0}^{\ast}=0,\quad\quad\;\,&\mathrm{on}\;\,\partial D_{1}^{\ast},\\
u_{0}^{\ast}=\varphi(x),\quad\;\,&\mathrm{on}\;\,\partial D.
\end{cases}
\end{align}

Under the following assumptions:
\begin{enumerate}
{\it
\item[$(\mathbf{ \Phi1})$]
for $m\geq n+1$, $2\leq k<m-n+1$, if
\begin{align}\label{KHMDC1}
h_{1}(x')-h(x')=|x'|^{m},\quad\mbox{in}\;B'_{R},
\end{align}
and
\begin{align}
\varphi^{i}(x)=\eta|x'|^{k},\quad\mbox{on}\;\Gamma^{-}_{R},\quad\mbox{for}~ i=1,2,\cdots,n;
\end{align}
\item[$(\mathbf{\Phi2})$]
for $m\geq n$, $k=1$, if
\begin{align*}
h_{1}(x')-h(x')=(1+x_{1})|x'|^{m},\quad\mbox{in}\;B'_{R},
\end{align*}
and
\begin{align*}
\varphi^{i}(x)=x_{1},\quad\mbox{on}\;\Gamma^{-}_{R},\quad\mbox{for}~i=1,2,\cdots,n;
\end{align*}
\item[$(\mathbf{\Phi3})$]
for $m=n$ and $k>1$, if (\ref{zw11}) holds and there exists some integer $1\leq k_{0}\leq n$ such that $Q^{\ast}_{k_{0}}[\varphi]\neq0$;
\item[$(\mathbf{\Phi4})$]
for $n-1\leq m<n$, there exists some integer $1\leq k_{0}\leq n$ such that $Q^{\ast}_{k_{0}}[\varphi]\neq0$;
\item[$(\mathbf{\Phi5})$]
for $m<n-1$, there exists some integer $1\leq k_{0}\leq n$ such that $Q^{\ast}_{k_{0}}[\varphi]\neq0$ and $Q^{\ast}_{\beta}[\varphi]=0$ for all $\beta\neq k_{0}$,}
\end{enumerate}
we obtain the lower bound of $|\nabla u|$ at the shortest line $\{x'=0'\}\cap\Omega$ between $D$ and $D_{1}$ as follows.
\begin{theorem}\label{MMM896}
Assume that $D_{1}\subset D\subseteq\mathbb{R}^{n}\,(n\geq2)$ are defined as above and conditions ({\bf{\em H1}})--({\bf{\em H5}}) hold with $\Sigma'=\{0'\}$. Let $u\in H^{1}(D;\mathbb{R}^{n})\cap C^{1}(\overline{\Omega};\mathbb{R}^{n})$ be the solution of (\ref{La.002}). Among $(\mathbf{\Phi1})$--$(\mathbf{\Phi5})$, if one of them holds, then for a sufficiently small $\varepsilon>0$,
\begin{align}\label{PPN1}
|\nabla u(x)|\geq\frac{\eta\rho_{k;0}(n,m;\varepsilon)}{C\varepsilon},\quad\quad\mathrm{for}\;\,x\in\{x'=0'\}\cap\Omega,
\end{align}
where $\rho_{k;0}(n,m;\varepsilon)=\rho_{k}(n,m;\varepsilon)/\rho_{0}(n,m;\varepsilon)$.

\end{theorem}

For the lower bound of $|\nabla u|$ on the cylinder surface $\{|x'|=\sqrt[m]{\varepsilon}\}\cap\Omega$, see Theorem \ref{AAA666666} below.

Our paper is organized as follows. Section $2$ provides a decomposition of the solution $u$ of (\ref{La.002}) as a linear combination of $u_{\alpha}$, $\alpha=0,1,\cdots,\frac{n(n+1)}{2}$, defined by (\ref{Lek2.012}) and (\ref{P2.005}), which reduces the proof of Theorem \ref{Lthm066} to estimates of $|\nabla u_{\alpha}|$ and $C^{\alpha}$, $\alpha=1,2,\cdots,\frac{n(n+1)}{2}$. Section $3$ presents the proof of Theorem \ref{Lthm066}. In section $4$ we give a classification for the given boundary data $\varphi$ according to its parity and then identify the singularities of blowup factors $Q_{\beta}[\varphi]$, $\beta=1,2,\cdots,\frac{n(n+1)}{2}$, as shown in Proposition \ref{lekm323}. The lower bound of $|\nabla u|$ at the shortest line $\{x'=0'\}\cap\Omega$ between $D$ and $D_{1}$ in Theorem \ref{MMM896} is proved in Section $5$. In Section $6$ we also show the lower bound of $|\nabla u|$ on the cylinder surface $\{|x'|=\sqrt[m]{\varepsilon}\}\cap\Omega$ in Theorem \ref{AAA666666}.

\section{Preliminary}
\subsection{Properties of the tensor $\mathbb{C}$}
For the isotropic elastic material, set
$$\mathbb{C}:=(C_{ijkl})=(\lambda\delta_{ij}\delta_{kl}+\mu(\delta_{ik}\delta_{jl}+\delta_{il}\delta_{jk})),\quad \mu>0,\quad n\lambda+2\mu>0.$$
The components $C_{ijkl}$ possess symmetry property:
\begin{align}\label{symm}
C_{ijkl}=C_{klij}=C_{klji},\quad i,j,k,l=1,2,...,n.
\end{align}
For $n\times n$ matrices $A=(A_{ij})$, $B=(B_{ij})$, denote
\begin{align*}
(\mathbb{C}A)_{ij}=\sum_{k,l=1}^{n}C_{ijkl}A_{kl},\quad\hbox{and}\quad(A,B)\equiv A:B=\sum_{i,j=1}^{n}A_{ij}B_{ij},
\end{align*}
Thus,
$$(\mathbb{C}A, B)=(A, \mathbb{C}B).$$
For an $n\times n$ real symmetric matrix $\xi=(\xi_{ij})$, it follows from (\ref{symm}) that
\begin{align}\label{ellip}
\min\{2\mu, n\lambda+2\mu\}|\xi|^2\leq(\mathbb{C}\xi, \xi)\leq\max\{2\mu, n\lambda+2\mu\}|\xi|^2,
\end{align}
where $|\xi|^2=\sum\limits_{ij}\xi_{ij}^2.$ That is, $\mathbb{C}$ satisfies the ellipticity condition. In particular,
\begin{align}\label{Le2.012}
\min\{2\mu, n\lambda+2\mu\}|A+A^T|^2\leq(\mathbb{C}(A+A^T), (A+A^T)).
\end{align}
Besides, it is known that for any open set $O$ and $u, v\in C^2(O;\mathbb{R}^{n})$,
\begin{align}\label{Le2.01222}
\int_O(\mathbb{C}^0e(u), e(v))\,dx=-\int_O\left(\mathcal{L}_{\lambda, \mu}u\right)\cdot v+\int_{\partial O}\frac{\partial u}{\partial \nu_0}\Big|_{+}\cdot v.
\end{align}
\subsection{Solution split}
As in \cite{BJL2017}, the solution $u$ of (\ref{La.002}) can be split as follows
\begin{align*}
u=\sum^{\frac{n(n+1)}{2}}_{\alpha=n}C^{\alpha}u_{\alpha}+u_{0},\quad\;\,\mathrm{in}\;\,\Omega,
\end{align*}
where $C^{\alpha},\,\alpha=1,2,\cdots,\frac{n(n+1)}{2},$ are some free constants to be determined later by the forth line in (\ref{La.002}), and $u_{\alpha}\in C^{1}(\overline{\Omega};\mathbb{R}^{n})\cap C^{2}(\Omega;\mathbb{R}^{n}),\,\alpha=1,2,\cdots,\frac{n(n+1)}{2},$ respectively, verifies
\begin{equation}\label{P2.005}
\begin{cases}
\mathcal{L}_{\lambda,\mu}u_{\alpha}=0,\quad\;\,&\mathrm{in}\;\,\Omega,\\
u_{\alpha}=\psi_{\alpha},\quad\;\,&\mathrm{on}\;\,\partial D_{1},\\
u_{\alpha}=0,\quad\;\,&\mathrm{on}\;\,\partial D,
\end{cases}
\end{equation}
and $u_{0}$ is defined by (\ref{Lek2.012}).
Therefore
\begin{align}\label{Le2.015}
\nabla u=\sum^{\frac{n(n+1)}{2}}_{\alpha=1}C^{\alpha}\nabla u_{\alpha}+\nabla u_{0},\quad\;\,\mathrm{in}\;\,\Omega.
\end{align}
Thus, in order to estimate $|\nabla u|$, it suffices to establish the following two aspects of estimates:

(i) estimates of $|\nabla u_{\alpha}|$, $\alpha=0,1,\cdots,\frac{n(n+1)}{2}$;

(ii) estimates of $C^{\alpha},$ $\alpha=1,2,\cdots,\frac{n(n+1)}{2}$.

\subsection{A general boundary value problem}

To estimate $|\nabla u_{\alpha}|$, $\alpha=1,2,\cdots,\frac{n(n+1)}{2}$, we begin with considering the following general boundary value problem:
\begin{equation}\label{P2.008}
\begin{cases}
\mathcal{L}_{\lambda,\mu}v:=\nabla\cdot(\mathbb{C}^{0}e(v))=0,\quad\;\,&\mathrm{in}\;\,\Omega,\\
v=\psi(x),&\mathrm{on}\;\,\partial D_{1},\\
v=0,&\mathrm{on}\;\,\partial D,
\end{cases}
\end{equation}
where $\psi(x)=(\psi^{1}(x),\psi^{2}(x),\cdots,\psi^{n}(x))\in C^{2}(\partial D_{1};\mathbb{R}^{n})$ is a given vector-valued function.

Note that the solution of (\ref{P2.008}) can be split into
$$v=\sum^{n}_{i=1}v_{i},$$
where $v_{i}=(v_{i}^{1},v_{i}^{2},\cdots,v_{i}^{n})^{T}$, $i=1,2,\cdots,n$, with $v_{i}^{j}=0$ for $j\neq i$, and $v_{i}$ solves the following boundary value problem
\begin{align}\label{P2.010}
\begin{cases}
  \mathcal{L}_{\lambda,\mu}v_{i}:=\nabla\cdot(\mathbb{C}^0e(v_{i}))=0,\quad&
\hbox{in}\  \Omega,  \\
v_{i}=( 0,\cdots,0,\psi^{i}, 0,\cdots,0)^{T},\ &\hbox{on}\ \partial{D}_{1},\\
v_{i}=0,&\hbox{on} \ \partial{D}.
\end{cases}
\end{align}
Then
\begin{align}\label{Le2.018}
\nabla v=\sum^{n}_{i=1}\nabla v_{i}.
\end{align}

We next estimate $|\nabla v_{i}|$ one by one. For this purpose, we define a scalar auxiliary function $\bar{v}\in C^{2}(\mathbb{R}^{n})$ satisfying $\bar{v}=1$ on $\partial D_{1}$, $\bar{v}=0$ on $\partial D$ and
\begin{align}\label{Le2.019}
\bar{v}(x)=\frac{x_{n}-h(x')}{\varepsilon+h_{1}(x')-h(x')},\quad\;\,\mathrm{in}\;\,\Omega_{2R},
\end{align}
and
\begin{align}\label{Le2.020}
\|\bar{v}\|_{C^{2}(\Omega\setminus\Omega_{R})}\leq C.
\end{align}
Extend $\psi\in C^{2}(\partial D_{1};\mathbb{R}^{n})$ to $\psi\in C^{2}(\overline{\Omega};\mathbb{R}^{n})$ such that $\|\psi^{i}\|_{C^{2}(\overline{\Omega\setminus\Omega_{R}})}\leq C\|\psi^{i}\|_{C^{2}(\partial D_{1})},$ for $i=1,2,\cdots,n$. Construct a cutoff function $\rho\in C^{2}(\overline{\Omega})$ satisfying $0\leq\rho\leq1$, $|\nabla\rho|\leq C$ on $\overline{\Omega}$, and
\begin{align}\label{Le2.021}
\rho=1\;\,\mathrm{on}\;\,\Omega_{\frac{3}{2}R},\quad\rho=0\;\,\mathrm{on}\;\,\overline{\Omega}\setminus\Omega_{2R}.
\end{align}
For $x\in\Omega$, we define
\begin{align*}
\tilde{v}_{i}(x)=\left(0,\cdots,0,[\rho(x)\psi^{i}(x',\varepsilon+h_{1}(x'))+(1-\rho(x))\psi^{i}(x)]\bar{v}(x),0,\cdots,0\right)^{T}.
\end{align*}
In particular,
\begin{align*}
\tilde{v}_{i}(x)=(0,\cdots,0,\psi^{i}(x',\varepsilon+h_{1}(x'))\bar{v}(x),0,\cdots,0)^{T}\quad\;\,\mathrm{in}\;\,\Omega_{R}
\end{align*}
and in light of (\ref{Le2.020}),
\begin{align*}
\|\tilde{v}_{i}\|_{C^{2}(\Omega\setminus\Omega_{R})}\leq C\|\psi^{i}\|_{C^{2}(\partial D_{1})}.
\end{align*}

For $(x',x_{n})\in\Omega_{2R}$, denote
\begin{align*}
\delta(x'):=\varepsilon+h_{1}(x')-h(x').
\end{align*}
A direct calculation yields that for $x\in\Omega_{R}$,
\begin{align}\label{Le2.0233}
\frac{|\psi^{i}(x',\varepsilon+h_{1}(x'))|}{C(\varepsilon+d^{m}(x'))}\leq|\nabla\tilde{v}_{i}|\leq\frac{C|\psi^{i}(x',\varepsilon+h_{1}(x'))|}{\varepsilon+d^{m}(x')}+C\|\nabla\psi^{i}\|_{L^{\infty}(\partial D_{1})}.
\end{align}
Similarly as in \cite{HL,BJL2017}, we can obtain pointwise gradient estimates for problem (\ref{P2.010}).
\begin{theorem}\label{thm86}
Assume as above. Let $v\in H^{1}(\Omega;\mathbb{R}^{n})$ be a weak solution of (\ref{P2.008}). Then for a sufficiently small $\varepsilon>0$,
\begin{align}\label{Le2.023}
|\nabla(v_{i}-\tilde{v}_{i})(x)|\leq\frac{C|\psi^{i}(x',\varepsilon+h_{1}(x'))|}{\sqrt[m]{\varepsilon+d^{m}(x')}}+C\|\psi^{i}\|_{C^{2}(\partial D_{1})},\quad\mathrm{in}\;\,\Omega_{R}.
\end{align}
Consequently, (\ref{Le2.023}), together with (\ref{Le2.0233}) yields that for $x\in\Omega_{R}$,
\begin{align}\label{Le2.025}
\frac{|\psi^{i}(x',\varepsilon+h_{1}(x'))|}{C(\varepsilon+d^{m}(x'))}\leq|\nabla v_{i}|\leq\frac{C|\psi^{i}(x',\varepsilon+h_{1}(x'))|}{\varepsilon+d^{m}(x')}+C\|\psi^{i}\|_{C^{2}(\partial D_{1})}.
\end{align}
Finally,
\begin{align*}
|\nabla v|\leq\frac{C|\psi(x',\varepsilon+h_{1}(x'))|}{\varepsilon+d^{m}(x')}+C\|\psi\|_{C^{2}(\partial D_{1})},\quad\;\,\mathrm{in}\;\,\Omega_{R},
\end{align*}
and
$$\|\nabla v\|_{L^{\infty}(\Omega\setminus\Omega_{R})}\leq C\|\psi\|_{C^{2}(\partial D_{1})}.$$
\end{theorem}
We remark that \eqref{Le2.023} is an improvement of that in \cite{HL}, where the denominator $\sqrt{\varepsilon+d^{m}(x')}$ is improved to $\sqrt[m]{\varepsilon+d^{m}(x')}$. The sketched proof of Theorem \ref{thm86} is given in the appendix.

\section{Proof of Theorem \ref{Lthm066}}

\subsection{Estimates of $|\nabla u_{\alpha}|$}
For $\alpha=1,2,\cdots,\frac{n(n+1)}{2}$, set
$$\psi=\psi_{\alpha},\;\tilde{u}_{\alpha}:=\psi_{\alpha}\bar{v}.$$
Applying Theorem \ref{thm86} to $u_{\alpha}$, we obtain
\begin{corollary}\label{co.001}
Assume as in Theorem \ref{Lthm066}. Let $u_{\alpha},\,\alpha=1,2,\cdots,\frac{n(n+1)}{2}$, be the weak solutions of (\ref{P2.005}), respectively. Then for a sufficiently small $\varepsilon>0$, $x\in\Omega_{R}$,
\begin{align}
|\nabla(u_{\alpha}-\tilde{u}_{\alpha})(x)|&\leq\frac{C}{\sqrt[m]{\varepsilon+d^{m}(x')}},\quad\;\,\alpha=1,2,\cdots,n,\label{Le2.0028}\\
|\nabla(u_{\alpha}-\tilde{u}_{\alpha})(x)|&\leq\frac{C|x'|}{\sqrt[m]{\varepsilon+d^{m}(x')}},\quad\;\,\alpha=n+1,\cdots,\frac{n(n+1)}{2}.\label{Lea2.0028}
\end{align}
Thus, combining with the definition of $\tilde{u}_{\alpha}$, we have
\begin{align}
\frac{1}{C(\varepsilon+d^{m}(x'))}\leq|\nabla u_{\alpha}|\leq&\frac{C}{\varepsilon+d^{m}(x')},\quad\alpha=1,2,\cdots,n,\label{Le2.029}\\
|\nabla u_{\alpha}|\leq&\frac{C(\varepsilon+|x'|)}{\varepsilon+d^{m}(x')},\quad\alpha=n+1,\cdots,\frac{n(n+1)}{2},\label{Le2.0291}
\end{align}
and
\begin{align}
|\nabla_{x'}u_{\alpha}|\leq&\frac{C}{\sqrt[m]{\varepsilon+d^{m}(x')}},\quad\alpha=1,2,\cdots,n,\label{QWE0.01}\\
|\nabla_{x'}u_{\alpha}|\leq&\frac{C|x'|}{\sqrt[m]{\varepsilon+d^{m}(x')}},\quad\alpha=n+1,\cdots,\frac{n(n+1)}{2},\label{QWE0.02}
\end{align}
and for $\alpha=1,2,\cdots,\frac{n(n+1)}{2}$,
\begin{align}\label{Le2.0292}
\|\nabla u_{\alpha}\|_{L^{\infty}(\Omega\setminus\Omega_{R})}\leq C.
\end{align}

\end{corollary}

For problem (\ref{Lek2.012}), the solution $u_{0}$ can be split as
$$u_{0}=\sum_{l=1}^{n}u_{0l},$$
where $u_{0l},~l=1,2,\cdots,n$, satisfy
\begin{align}\label{RTP101}
\begin{cases}
\mathcal{L}_{\lambda,\mu}u_{0l}=0,\quad&
\hbox{in}\  \Omega,  \\
u_{0l}=0&\hbox{on}\ \partial{D}_{1},\\
u_{0l}=(0,\cdots,0,\varphi^{l}(x),0,\cdots,0)^{T},&\hbox{on} \ \partial{D}.
\end{cases}
\end{align}
For $x\in\Omega$, define
\begin{align*}
\tilde{u}_{0l}(x):=(0,\cdots,0,[\rho(x)\varphi^{l}(x',h(x'))+(1-\rho(x))\varphi^{l}(x)](1-\bar{v}(x)),0,\cdots,0)^{T},
\end{align*}
where $\rho\in C^{2}(\overline{\Omega})$ is defined by (\ref{Le2.021}). Then,
\begin{align*}
\tilde{u}_{0l}(x):=(0,\cdots,0,\varphi^{l}(x',h(x'))(1-\bar{v}(x)),0,\cdots,0)^{T},\quad\;\,\mathrm{in}\;\,\Omega_{R}.
\end{align*}

Likewise, applying Theorem \ref{thm86}, we derive the following estimates of $|\nabla u_{0k}|$ for the problem (\ref{RTP101}).
\begin{lemma}\label{lem89}
Assume as in Theorem \ref{Lthm066}. Let $u_{0l},~l=1,2,\cdots,n$ be the weak solution of (\ref{RTP101}). Assume that $\|\varphi\|_{C^{2}(\partial D)}>0$. Then for a sufficiently small $\varepsilon>0$, $l=1,2,\cdots,n$,
\begin{align}\label{RTP1311}
|\nabla(u_{0l}-\tilde{u}_{0l})(x)|\leq\frac{C|\varphi^{l}(x',h(x'))|}{\sqrt[m]{\varepsilon+d^{m}(x')}}+C\|\varphi^{l}\|_{C^{2}(\partial D)},\quad\;\,\mathrm{in}\;\,\Omega_{R}.
\end{align}
Consequently, for $x\in\Omega_{R}$,
\begin{align}\label{HT1231}
|\nabla_{x'}u_{0l}|\leq\frac{C|\varphi^{l}(x',h(x'))|}{\sqrt[m]{\varepsilon+d^{m}(x')}},
\end{align}
and
\begin{align}\label{Le2.026}
\frac{|\varphi^{l}(x',h(x'))|}{C(\varepsilon+d^{m}(x'))}\leq|\partial_{n}u_{0l}|\leq&\frac{C|\varphi^{l}(x',h(x'))|}{\varepsilon+d^{m}(x')}+C\|\varphi^{l}\|_{C^{2}(\partial D)},
\end{align}
and
\begin{align}\label{Le2.027}
\|\nabla u_{0l}\|_{L^{\infty}(\Omega\setminus\Omega_{R})}\leq C\|\varphi^{l}\|_{C^{2}(\partial D)}.
\end{align}

\end{lemma}
The proof is similar to previous one and thus omitted.
\subsection{Estimates of $C^{\alpha},\;\alpha=1,2,\cdots,\frac{n(n+1)}{2}$.}

For $\alpha,\beta=1,2,\cdots,\frac{n(n+1)}{2}$, denote
$$a_{\alpha\beta}:=-\int_{\partial D_{1}}\frac{\partial u_{\alpha}}{\partial\nu_{0}}\Big|_{+}\cdot\psi_{\beta},\quad Q_{\beta}[\varphi]:=\int_{\partial D_{1}}\frac{\partial u_{0}}{\partial\nu_{0}}\Big|_{+}\cdot\psi_{\beta}.$$
Multiplying the first line of (\ref{Lek2.012}) and (\ref{P2.005}) by $u_{\beta}$, respectively, and integrating by parts over $\Omega$ yields
\begin{align}\label{BWN654}
a_{\alpha\beta}=\int_{\Omega}(\mathbb{C}^{0}e(u_{\alpha}),e(u_{\beta}))\,dx,\quad Q_{\beta}[\varphi]=-\int_{\Omega}(\mathbb{C}^{0}e(u_{0}),e(u_{\beta}))\,dx.
\end{align}
For $0<r\leq2R$, denote the top boundary of $\Omega_{r}$ by
\begin{align*}
\Gamma^{+}_{r}=\left\{x\in\mathbb{R}^{n}\big|\,x_{n}=\varepsilon+h_{1}(x'),\;|x'|<r\right\}.
\end{align*}

Before estimating $a_{\alpha\alpha}$, $\alpha=1,2,\cdots,\frac{n(n+1)}{2}$, we state a Lemma, its proof referred to \cite{BJL2017}.
\begin{lemma}\label{FBC6}
There exists a positive universal constant $C$, independent of $\varepsilon$, such that
\begin{align*}
\sum^{\frac{n(n+1)}{2}}_{\alpha,\beta=1}a_{\alpha\beta}\xi_{\alpha}\xi_{\beta}\geq\frac{1}{C},\quad\;\,\forall\;\xi\in\mathbb{R}^{\frac{n(n+1)}{2}},\;|\xi|=1.
\end{align*}
\end{lemma}

\begin{lemma}\label{lemmabc}
For $n\geq2$, we obtain that for $\alpha=1,2,\cdots,n$,
\begin{align}
\frac{|\Sigma'|+\varepsilon\rho_{0}(n,m;\varepsilon)}{C\varepsilon}\leq&a_{\alpha\alpha}\leq\frac{C(|\Sigma'|+\varepsilon\rho_{0}(n,m;\varepsilon))}{\varepsilon};\label{Le2.033}
\end{align}
for $\alpha=n+1,\cdots,\frac{n(n+1)}{2}$,
\begin{align}
\frac{|\Sigma'|^{\frac{n+1}{n-1}}+\varepsilon\rho_{2}(n,m;\varepsilon)}{C\varepsilon}\leq&a_{\alpha\alpha}\leq\frac{C(|\Sigma'|^{\frac{n+1}{n-1}}+\varepsilon\rho_{2}(n,m;\varepsilon))}{\varepsilon};\label{Le2.035}
\end{align}
and, for $\alpha,\beta=1,2,\cdots,\frac{n(n+1)}{2},\;\alpha\neq\beta$,
\begin{align}
|a_{\alpha\beta}|\leq&C\rho_{n}(\varepsilon,|\Sigma'|),\label{Le2.036}
\end{align}
where
\begin{align}\label{Le2.0366}
\rho_{n}(\varepsilon,|\Sigma'|)=
\begin{cases}
\frac{|\Sigma'|}{\sqrt[m]{\varepsilon}}+|\ln\varepsilon|,\quad\;\,&n=2,\\
\frac{|\Sigma'|}{\sqrt[m]{\varepsilon}}+|\Sigma'|^{\frac{n-2}{n-1}}|\ln\varepsilon|+1,&n\geq3.
\end{cases}
\end{align}
\end{lemma}
\begin{proof}
{\bf Step 1.} Estimate of $a_{\alpha\alpha}$, $\alpha=1,2,\cdots,n$. If $\Sigma'=B'_{r}(0')$, on one hand, it follows from (\ref{ellip}) and (\ref{Le2.0292}) that
\begin{align}
a_{\alpha\alpha}=&\int_{\Omega}(\mathbb{C}^{0}e(u_{\alpha}),e(u_{\alpha}))\,dx\leq C\int_{\Omega}|\nabla u_{\alpha}|^{2}dx\notag\\
\leq&C\left(\int_{B_{r}'}\frac{dx'}{\varepsilon}+\int_{B_{R}'\setminus B'_{r}}\frac{dx'}{\varepsilon+(|x'|-r)^{m}}\right)+C\notag\\
\leq&\frac{C(|\Sigma'|+\sqrt[m]{\varepsilon}|\Sigma'|^{\frac{n-2}{n-1}}+\varepsilon\rho_{0}(n,m;\varepsilon))}{\varepsilon}.\label{MMNMN1}
\end{align}

On the other hand, from the definition of $\tilde{u}_{\alpha}$ and (\ref{Le2.0028}), we have
\begin{align*}
a_{\alpha\alpha}\geq\frac{1}{C}\int_{\Omega}|e(u_{\alpha})|^{2}dx
\geq&\frac{1}{2C}\int_{\Omega_{R}}|e(\tilde{u}_{\alpha})|^{2}dx-C\int_{\Omega_{R}}|e(u_{\alpha}-\tilde{u}_{\alpha})|^{2}dx\\
\geq&\frac{1}{2C}\int_{\Omega_{R}}|\partial_{n}\bar{v}|^{2}dx-C\\
\geq&\frac{|\Sigma'|+\sqrt[m]{\varepsilon}|\Sigma'|^{\frac{n-2}{n-1}}+\varepsilon\rho_{0}(n,m;\varepsilon)}{C\varepsilon}.
\end{align*}

If $\Sigma'=\{0'\}$, by using Lemma \ref{FBC6} and (\ref{MMNMN1}), for $m<n-1$, we have
\begin{align*}
\frac{1}{C}\leq a_{\alpha\alpha}\leq C,
\end{align*}
The case for $m\geq n-1$ is the same as before. Therefore, (\ref{Le2.033}) is established.

{\bf Step 2.} Estimate of $a_{\alpha\alpha}$, $\alpha=n+1,\cdots,\frac{n(n+1)}{2}$. If $\Sigma'=B'_{r}(0')$, then
\begin{align}
|a_{\alpha\alpha}|\leq\,C\int_{\Omega}|\nabla u_{\alpha}|^{2}dx
\leq&C\left(\int_{B_{r}'}\frac{|x'|^{2}}{\varepsilon}\,dx'+\int_{B_{R}'\setminus B'_{r}}\frac{|x'|^{2}}{\varepsilon+(|x'|-r)^{m}}\right)+C\notag\\
\leq&\frac{C(|\Sigma'|^{\frac{n+1}{n-1}}+\sqrt[m]{\varepsilon}|\Sigma'|^{\frac{n}{n-1}}+\varepsilon\rho_{2}(n,m;\varepsilon))}{\varepsilon}.\label{MMNMN3}
\end{align}

On the other hand, it is easy to see that there exist two indexes $j$ and $k$ such that $1\leq j<k\leq n$ and $\tilde{u}_{\alpha}^{k}=x_{j}\bar{v}$. Then, utilizing (\ref{Lea2.0028}), we obtain
\begin{align*}
|a_{\alpha\alpha}|\geq\frac{1}{C}\int_{\Omega}|e(u_{\alpha})|^{2}dx
\geq&\frac{1}{2C}\int_{\Omega_{R}}|e(\tilde{u}_{\alpha})|^{2}dx-C\int_{\Omega_{R}}|e(u_{\alpha}-\tilde{u}_{\alpha})|^{2}dx\\
\geq&\frac{1}{2C}\int_{\Omega_{R}}|x_{j}\partial_{n}\bar{v}|^{2}dx-C\\
\geq&\frac{|\Sigma'|^{\frac{n+1}{n-1}}+\sqrt[m]{\varepsilon}|\Sigma'|^{\frac{n}{n-1}}+\varepsilon\rho_{2}(n,m;\varepsilon)}{C\varepsilon}.
\end{align*}

If $\Sigma'=\{0'\}$, in light of Lemma \ref{FBC6} and (\ref{MMNMN3}), for $m<n+1$,
\begin{align*}
\frac{1}{C}\leq a_{\alpha\alpha}\leq C.
\end{align*}
The case for $m\geq n+1$ is the same as above. Thus, (\ref{Le2.035}) holds.

{\bf Step 3.} Estimate of $a_{\alpha\beta}$, $\alpha,\beta=1,2,\cdots,n,\,\alpha\neq\beta$.

By symmetry, we only need to consider the case of $\alpha=1,2,\cdots,n,$ $\beta=1,\cdots,n-1$, provided that $\alpha\neq\beta$.
\begin{align*}
a_{\alpha\beta}=-\int_{\partial D_{1}}\frac{\partial u_{\alpha}}{\partial\nu_{0}}\Big|_{+}\cdot\psi_{\beta}
=&-\lambda\int_{\Gamma_{R}^{+}}\sum_{k=1}^n\partial_{k}u_{\alpha}^{k}\nu_\beta
-\mu\int_{\Gamma_{R}^{+}}\sum_{j=1}^{n-1}\left(\partial_{j}u_{\alpha}^{\beta}+\partial_{\beta}u_{\alpha}^{j}\right)\nu_j\\
&-\mu\int_{\Gamma_{R}^{+}}\partial_{n}u^{\beta}_{\alpha}\nu_{n}-\mu\int_{\Gamma_{R}^{+}}\partial_{\beta}u_{\alpha}^{n}\nu_{n}+O(1),\\
=&:\lambda\mathrm{I}_{\alpha\beta}^{1}+\mu(\mathrm{I}_{\alpha\beta}^{2}+\mathrm{I}_{\alpha\beta}^{3}+\mathrm{I}_{\alpha\beta}^{4})+O(1),\\
\end{align*}
where
$$\nu=\frac{(\nabla_{x'}h_{1}(x'),-1)}{\sqrt{1+|\nabla_{x'}h_{1}(x')|^{2}}}.$$
In view of ({\bf{\em H3}}), we know that for $i=1,\cdots,n-1$,
\begin{align}\label{AB1.01}
|\nu_{i}|\leq Cd^{m-1}(x'),\quad|\nu_{n}|\leq1.
\end{align}
Thus, using (\ref{Le2.029}) and (\ref{AB1.01}),
\begin{align*}
|\mathrm{I}_{\alpha\beta}^{1}|\leq&\int_{\Gamma_{R}^{+}}\left|\sum_{k=1}^n\partial_{k}u_{\alpha}^{k}\nu_\beta\right|\leq\int_{\Gamma^{+}_{R}}\frac{Cd^{m-1}(x')}{\varepsilon+d^{m}(x')}\leq C\rho_{n}(\varepsilon,|\Sigma'|),
\end{align*}
where $\rho_{n}(\varepsilon,|\Sigma'|)$ is defined by (\ref{Le2.0366}). It follows from (\ref{QWE0.01}) and (\ref{AB1.01}) that
\begin{align*}
|\mathrm{I}_{\alpha\beta}^{2}|\leq&\int_{\Gamma_{R}^{+}}\sum_{j=1}^{n-1}\left|\left(\partial_{j}u_{\alpha}^{\beta}+\partial_{\beta}u_{\alpha}^{j}\right)\nu_j\right|\leq\int_{B'_{R}}\frac{Cd^{m-1}(x')}{(\varepsilon+d^{m}(x'))^{\frac{1}{m}}}\leq C,
\end{align*}
and
\begin{align*}
|\mathrm{I}_{\alpha\beta}^{4}|\leq&\int_{\Gamma_{R}^{+}}\left|\partial_{\beta}u_{\alpha}^{n}\nu_{n}\right|\leq\int_{B'_{R}}\frac{C}{(\varepsilon+d^{m}(x'))^{\frac{1}{m}}}\leq C\rho_{n}(\varepsilon,|\Sigma'|),
\end{align*}
while, by using (\ref{Le2.0028}) and $\tilde{u}^{\beta}_{\alpha}=0$, we obtain
\begin{align*}
|\mathrm{I}^{3}_{\alpha\beta}|\leq&\left|\int_{\Gamma^{+}_{R}}\partial_{n}u_{\alpha}^{\beta}\right|\leq\left|\int_{\Gamma^{+}_{R}}\partial_{n}(u_{\alpha}-\tilde{u}_{\alpha})^{\beta}\right|
\leq\int_{B_{R}'}\frac{C}{(\varepsilon+d^{m}(x'))^{\frac{1}{m}}}\leq C\rho_{n}(\varepsilon,|\Sigma'|).
\end{align*}
Consequently,
\begin{align}\label{LYW20}
|a_{\alpha\beta}|\leq C\rho_{n}(\varepsilon,|\Sigma'|).
\end{align}
{\bf Step 4.} Estimate of $a_{\alpha\beta}$, $\alpha=1,2,\cdots,n,\,\beta=n+1,\cdots,\frac{n(n+1)}{2},\alpha\neq\beta$.

We here take the case when $\alpha=1,\,\beta=n+1$ for example. Since $\psi_{n+1}=(x_{2},-x_{1},0,\cdots,0)^{T}$, making use of the boundedness of $|\nabla u_{\alpha}|$ on $\partial D_{1}\setminus\Gamma^{+}_{R}$, we obtain
\begin{align*}
a_{1,n+1}=&-\int_{\Gamma^{+}_{R}}\bigg(\lambda\sum^{n}_{k=1}\partial_{k}u_{1}^{k}\nu_{1}+\mu\sum^{n}_{j=1}\big(\partial_{1}u_{1}^{j}+\partial_{j}u_{1}^{1}\big)\nu_{j}\bigg)x_{2}\\
&+\int_{\Gamma^{+}_{R}}\bigg(\lambda\sum^{n}_{k=1}\partial_{k}u_{1}^{k}\nu_{2}+\mu\sum^{n}_{j=1}\big(\partial_{2}u_{1}^{j}+\partial_{j}u_{1}^{2}\big)\nu_{j}\bigg)x_{1}+O(1),\\
=&:-\mathrm{I}_{1,n+1}+\mathrm{II}_{1,n+1}+O(1).
\end{align*}
As for $\mathrm{I}_{1,n+1}$, we split it into three parts
\begin{align*}
\mathrm{I}^{1}_{1,n+1}=&-\int_{\Gamma^{+}_{R}}\bigg(\lambda\sum^{n}_{k=1}\partial_{k}u_{1}^{k}\nu_{1}+\mu\sum^{n-1}_{j=1}\big(\partial_{1}u_{1}^{j}+\partial_{j}u_{1}^{1}\big)\nu_{j}\bigg)x_{2},\\
\mathrm{I}^{2}_{1,n+1}=&-\mu\int_{\Gamma^{+}_{R}}\partial_{1}u_{1}^{n}\nu_{n}x_{2},\\
\mathrm{I}^{3}_{1,n+1}=&-\mu\int_{\Gamma^{+}_{R}}\partial_{n}u_{1}^{1}\nu_{n}x_{2}.
\end{align*}
It follows from (\ref{Le2.029}) and (\ref{AB1.01}) that
\begin{align*}
|\mathrm{I}^{1}_{1,n+1}|\leq&\int_{\Gamma^{+}_{R}}\bigg|\bigg(\lambda\sum^{n}_{k=1}\partial_{k}u_{1}^{k}\nu_{1}+\mu\sum^{n-1}_{j=1}\big(\partial_{1}u_{1}^{j}+\partial_{j}u_{1}^{1}\big)\nu_{j}\bigg)x_{2}\bigg|\\
\leq&\int_{B'_{R}}\frac{Cd^{m-1}(x')}{\varepsilon+d^{m}(x')}\leq C\rho_{n}(\varepsilon,|\Sigma'|),
\end{align*}
while, in view of $\tilde{u}^{n}_{1}=0$ and (\ref{Le2.0028}),
\begin{align*}
|\mathrm{I}^{2}_{1,n+1}|\leq&\mu\int_{\Gamma^{+}_{R}}\left|\partial_{1}(u_{1}-\tilde{u}_{1})^{n}\nu_{n}x_{2}\right|\leq\int_{B'_{R}}\frac{C}{\sqrt[m]{\varepsilon+d^{m}(x')}}\leq C\rho_{n}(\varepsilon,|\Sigma'|).
\end{align*}
Note that
\begin{align}\label{KTL0.11}
|x_{n}|=|\varepsilon+h_{1}(x')|\leq C(\varepsilon+d^{m}(x')).
\end{align}
Then it follows from (\ref{Le2.0028}), (\ref{KTL0.11}) and the symmetry of $B'_{R}$ that
\begin{align*}
|\mathrm{I}^{3}_{1,n+1}|\leq&\mu\int_{\Gamma^{+}_{R}}\left|\partial_{n}(u_{1}-\tilde{u}_{1})^{1}\nu_{n}x_{2}\right|+\mu~\Big|\int_{B'_{R}}\frac{x_{2}}{\delta}\Big|\\
\leq&
\begin{cases}
\mu\int_{B'_{R}}\frac{Cd^{m}(x_{1})}{\sqrt[m]{\varepsilon+d^{m}(x_{1})}}+\mu\int_{B'_{R}}\frac{Cd^{m}(x_{1})}{\varepsilon+d^{m}(x_{1})},&\quad n=2,\\
\mu\int_{B'_{R}}\frac{C|x'|}{\sqrt[m]{\varepsilon+d^{m}(x')}},&\quad n\geq3
\end{cases}\\
\leq& C\rho_{n}(\varepsilon,|\Sigma'|).
\end{align*}
Thus,
$$|\mathrm{I}_{1,n+1}|\leq C\rho_{n}(\varepsilon,|\Sigma'|).$$

By the same argument, we get $|\mathrm{II}_{1,n+1}|\leq C\rho_{n}(\varepsilon,|\Sigma'|)$. So
\begin{align}\label{LYW21}
|a_{1,n+1}|\leq C\rho_{n}(\varepsilon,|\Sigma'|).
\end{align}

{\bf Step 5.} Estimate of $a_{\alpha\beta}$, $\alpha,\beta=n+1,\cdots,\frac{n(n+1)}{2},\,\alpha\neq\beta,\,n\geq3$.

Take the case that $\alpha=n+1,$ $\beta=n+2$ for example. The other cases are the same. Note that $\tilde{u}_{n+1}=(\bar{v}x_{2},-\bar{v}x_{1},0,\cdots,0)^{T}$ and $\psi_{n+2}=(x_{3},0,-x_{1},0,\cdots,0)^{T}$. Similarly as before, we have
\begin{align*}
a_{n+1,n+2}=&-\int_{\Gamma^{+}_{R}}\bigg(\lambda\sum^{n}_{k=1}\partial_{k}u_{n+1}^{k}\nu_{1}+\mu\sum^{n}_{j=1}\big(\partial_{1}u^{j}_{n+1}+\partial_{j}u_{n+1}^{1}\big)\nu_{j}\bigg)x_{3}\\
&+\int_{\Gamma^{+}_{R}}\bigg(\lambda\sum^{n}_{k=1}\partial_{k}u_{n+1}^{k}\nu_{3}+\mu\sum^{n}_{j=1}\big(\partial_{3}u_{n+1}^{j}+\partial_{j}u_{n+1}^{3}\big)\nu_{j}\bigg)x_{1}+O(1)\\
=&:-\mathrm{I}_{n+1,n+2}+\mathrm{II}_{n+1,n+2}+O(1).
\end{align*}

With regard to $\mathrm{I}_{n+1,n+2}$, it can be split
\begin{align*}
\mathrm{I}_{n+1,n+2}^{1}=&\int_{\Gamma^{+}_{R}}\bigg(\lambda\sum^{n}_{k=1}\partial_{k}u_{n+1}^{k}\nu_{1}+\mu\sum^{n}_{j=1}\partial_{1}u^{j}_{n+1}\nu_{j}+\mu\sum^{n-1}_{j=1}\partial_{j}u_{n+1}^{1}\nu_{j}\bigg)x_{3},
\end{align*}
and
\begin{align*}
\mathrm{I}_{n+1,n+2}^{2}=&\int_{\Gamma^{+}_{R}}\mu\partial_{n}u_{n+1}^{1}\nu_{n}x_{3}.
\end{align*}
Utilizing (\ref{QWE0.02}) and (\ref{AB1.01}), we have
\begin{align*}
|\mathrm{I}_{n+1,n+2}^{1}|\leq&\int_{B'_{R}}\frac{C}{\sqrt[m]{\varepsilon+d^{m}(x')}}\leq C\rho_{n}(\varepsilon,|\Sigma'|),
\end{align*}
while, by the symmetry and (\ref{Lea2.0028}),
\begin{align*}
|\mathrm{I}_{n+1,n+2}^{2}|\leq&\int_{\Gamma^{+}_{R}}\mu\left|\partial_{n}(u_{n+1}-\tilde{u}_{n+1})^{1}\nu_{n}x_{3}\right|+\mu\left|\int_{B'_{R}}\frac{x_{2}x_{3}}{\delta}\right|\leq C\rho_{n}(\varepsilon,|\Sigma'|).
\end{align*}
Thus,
$$|\mathrm{I}_{n+1,n+2}|\leq C\rho_{n}(\varepsilon,|\Sigma'|).$$
Similarly, we obtain $|\mathrm{II}_{n+1,n+2}|\leq C\rho_{n}(\varepsilon,|\Sigma'|)$. Hence,
\begin{align}\label{LYW22}
|a_{n+1,n+2}|\leq C\rho_{n}(\varepsilon,|\Sigma'|).
\end{align}

Combining {\bf Step 3}, {\bf 4} and {\bf 5}, (\ref{Le2.036}) holds.

\end{proof}

Before estimating $C^{\alpha},$ $\alpha=1,2,\cdots,\frac{n(n+1)}{2}$, we first state a result. Its proof is a slight modification of that in \cite{BLL2017}.
\begin{lemma}\label{lemma866}
For $n\geq1$, let $A,D$ be $n\times n$ invertible matrices and $B$ and $C$ be $n\times n$ matrices satisfying, for some $0<\theta<1$, $\gamma_{1},\gamma_{3}>1$, and $0<\gamma_{2}\leq1$,
\begin{align*}
\|A^{-1}\|\leq\frac{1}{\theta\gamma_{1}},\quad\;\,\|B\|+\|C\|\leq\frac{1}{\theta\gamma_{2}},\quad\|D^{-1}\|\leq\frac{1}{\theta\gamma_{3}}.
\end{align*}
Then there exist $\bar{\gamma}=\bar{\gamma}(n)>1$ and $C(n)>1$, such that if $\gamma_{1}\gamma_{2}^{2}\geq\frac{\bar{\gamma}(n)}{\theta^{4}}$,
\begin{gather*}
\begin{pmatrix} A&B \\  C&D
\end{pmatrix}
\end{gather*}
is invertible. Moreover,
\begin{align*}
\begin{pmatrix} E_{11}&E_{12} \\  E_{12}^{T}&E_{22}
\end{pmatrix}:=
\begin{pmatrix} A&B \\  C&D
\end{pmatrix}^{-1}-
\begin{pmatrix} A^{-1}&0 \\  0&D^{-1}
\end{pmatrix}
\end{align*}
satisfies
\begin{align*}
\|E_{11}\|\leq\frac{C(n)}{\theta^{5}\gamma^{2}_{1}\gamma^{2}_{2}},\quad\|E_{12}\|\leq\frac{C(n)}{\theta^{3}\gamma_{1}\gamma_{2}},\quad\mathrm{and}\quad\|E_{22}\|\leq\frac{C(n)}{\theta^{5}\gamma_{1}\gamma^{2}_{2}\gamma^{2}_{3}}.
\end{align*}

\end{lemma}

\begin{prop}\label{proposition123}
Let $C^{\alpha}$ and $\varphi$ be defined as before, $\alpha=1,2,\cdots,\frac{n(n+1)}{2}$. Then, for a sufficiently small $\varepsilon>0$,
\begin{align}\label{Le11.011}
|C^{\alpha}|\leq&\frac{C\varepsilon}{|\Sigma'|+\varepsilon\rho_{0}(n,m;\varepsilon)}Q_{\beta;\mathrm{I}}[\varphi],\quad\alpha=1,2,\cdots,n,
\end{align}
and
\begin{align}\label{Le11.012}
|C^{\alpha}|\leq\frac{C\varepsilon}{|\Sigma'|^{\frac{n+1}{n-1}}+\varepsilon\rho_{2}(n,m;\varepsilon)}Q_{\beta;\mathrm{II}}[\varphi],\quad\alpha=n+1,\cdots,\frac{n(n+1)}{2},
\end{align}
where $Q_{\beta;\mathrm{I}}[\varphi]$ and $Q_{\beta;\mathrm{II}}[\varphi]$ are defined by (\ref{NCE1}).

\end{prop}
\begin{proof}
Utilizing decomposition (\ref{Le2.015}) of $|\nabla u|$ and recalling the fourth line of (\ref{La.002}), we derive that
\begin{align}\label{Le3.078}
\sum^{\frac{n(n+1)}{2}}_{\alpha=1}C^{\alpha}a_{\alpha\beta}=Q_{\beta}[\varphi],\quad\beta=1,2,\cdots,\frac{n(n+1)}{2}.
\end{align}
Define
\begin{align*}
X^{1}=&\left(C^{1},\cdots,C^{n}\right)^{T},\quad X^{2}=\left(C^{n+1},\cdots,C^{\frac{n(n+1)}{2}}\right)^{T}\\
Y^{1}=&\left(Q_{1}[\varphi],Q_{2}[\varphi],\cdots,Q_{n}[\varphi]\right)^{T},\quad Y^{2}=\left(Q_{n+1}[\varphi],\cdots,Q_{\frac{n(n+1)}{2}}[\varphi]\right)^{T}
\end{align*}
and
\begin{gather*}A=\begin{pmatrix} a_{11}&\cdots&a_{1n} \\\\ \vdots&\ddots&\vdots\\\\a_{n1}&\cdots&a_{nn}\end{pmatrix}  ,\quad
B=\begin{pmatrix} a_{1,n+1}&\cdots&a_{1,\frac{n(n+1)}{2}} \\\\ \vdots&\ddots&\vdots\\\\a_{n,n+1}&\cdots&a_{n,\frac{n(n+1)}{2}}\end{pmatrix} ,\end{gather*}\begin{gather*}
D=\begin{pmatrix} a_{n+1,n+1}&\cdots&a_{n+1,\frac{n(n+1)}{2}} \\\\ \vdots&\ddots&\vdots\\\\a_{\frac{n(n+1)}{2},n+1}&\cdots&a_{\frac{n(n+1)}{2},\frac{n(n+1)}{2}} \end{pmatrix}.
\end{gather*}
Therefore, in light of the symmetry of $a_{\alpha\beta}$, (\ref{Le3.078}) can be rewritten as
\begin{gather*}
\begin{pmatrix} A&B \\  B^T&D
\end{pmatrix}
\begin{pmatrix}
X^{1}\\
X^{2}
\end{pmatrix}=
\begin{pmatrix}
Y^{1}\\
Y^{2}
\end{pmatrix}.
\end{gather*}

If $\Sigma'=\{0'\}$, by utilizing Lemma \ref{FBC6}, we deduce that the matrix
\begin{gather*}
\begin{pmatrix} A&B \\  B^T&D
\end{pmatrix}
\end{gather*}
is positive definite, thus invertible. For $m<n-1$, it follows from Lemma \ref{lemmabc} that
\begin{align*}
|X^{1}|\leq C.
\end{align*}
As for $m\geq n-1$, matrices $A,B,D$ satisfy the assumptions of Lemma \ref{lemma866} with $\gamma_{1}=\rho_{0}(n,m;\varepsilon)$, $\gamma_{2}=\frac{1}{\rho_{n}(\varepsilon,|\Sigma'|)}$, $\gamma_{3}=\rho_{2}(n,m;\varepsilon)$, and $\theta=\frac{1}{C}$. Then
\begin{align}\label{GMD1}
X^{1}=&\frac{1}{\det A}A^{\ast}Y^{1}+O\left(\left(\frac{\rho_{n}(\varepsilon,|\Sigma'|)}{\rho_{0}(n,m;\varepsilon)}\right)^{2}\right),
\end{align}
and
\begin{align}\label{GMD2}
X^{2}=&\frac{1}{\det D}D^{\ast}Y^{2}+O\left(\frac{[\rho_{n}(\varepsilon,|\Sigma'|)]^{2}}{\rho_{0}(n,m;\varepsilon)[\rho_{2}(n,m;\varepsilon)]^{2}}\right),
\end{align}
where $A^{\ast}$ and $D^{\ast}$ are the adjoint matrix of $A$ and $D$, respectively. Thus, in view of (\ref{GMD1}) and (\ref{GMD2}), we complete the proof of (\ref{Le11.011}) and (\ref{Le11.012}) if $\Sigma'=\{0'\}$.

By the same argument, we also obtain that (\ref{Le11.011}) and (\ref{Le11.012}) hold if $\Sigma'=B_{r}'(0')$.

\end{proof}

\subsection{Proof of Theorem \ref{Lthm066}}
Recalling decomposition (\ref{Le2.015}) for $|\nabla u|$ and making use of the results derived in Corollary \ref{co.001}, Lemma \ref{lem89} and Proposition \ref{proposition123}, we deduce that for $x\in\Omega_{R}$,
\begin{align*}
|\nabla u(x)|\leq&C\Bigg(\frac{Q_{\beta;\mathrm{I}}[\varphi]}{\varepsilon+d^{m}(x')}\frac{\varepsilon}{|\Sigma'|+\varepsilon\rho_{0}(n,m;\varepsilon)}+\frac{|\varphi(x',h(x'))|}{\varepsilon+d^{m}(x')}\notag\\
&+\frac{Q_{\beta;\mathrm{II}}[\varphi]}{\varepsilon+d^{m}(x')}\frac{(\varepsilon+|x'|)\varepsilon}{|\Sigma'|^{\frac{n+1}{n-1}}+\varepsilon\rho_{2}(n,m;\varepsilon)}+\|\varphi\|_{C^{2}(\partial D)}\Bigg).
\end{align*}
This completes the proof of Theorem \ref{Lthm066}.
\begin{remark}
We remark that Theorem \ref{Lthm066} also holds for a bounded convex and axisymmetric region $\Sigma'\subset\mathbb{R}^{n-1}$, such as $\Sigma'=\{x'\in\mathbb{R}^{n-1}|\sum^{n-1}_{i=1}\big(\frac{|x_{i}|}{a_{i}}\big)^{m}\leq1,~a_{i}>0,~m\geq2\}$. The detailed proof is left to the interested readers.
\end{remark}

\section{Proof of Proposition \ref{lekm323}}

%
\subsection{Proof of Proposition \ref{lekm323}}
{\bf Step 1.} Proof of (\ref{HND1112}) under ({\bf{\em A1}}). In light of (\ref{BWN654}), it follows from assumption ({\bf{\em A1}}), Corollary \ref{co.001} and Lemma \ref{lem89} that for $\beta=1,2,\cdots,n$,
\begin{align*}
|Q_{\beta}[\varphi]|=&\left|\int_{\Omega}(\mathbb{C}^{0}e(u_{0}),e(u_{\beta}))\,dx\right|
\leq\int_{\Omega}C|\nabla u_{0}||\nabla u_{\beta}|\,dx\\
\leq&\int_{\Omega_{R}}\frac{C|\varphi(x',h(x'))|}{(\varepsilon+|x'|^{m})^{2}}\,dx+C\|\varphi\|_{C^{2}(\partial D)}\\
\leq&C\eta\int_{|x'|<R}\frac{|x'|^{k}}{\varepsilon+|x'|^{m}}+C\|\varphi\|_{C^{2}(\partial D)},\\
\leq&C(\eta\rho_{k}(n,m;\varepsilon)+\|\varphi\|_{C^{2}(\partial D)}).
\end{align*}

Similarly, in light of assumption ({\bf{\em A2}}), we obtain that for $\beta=n+1,\cdots,\frac{n(n+1)}{2}$,
\begin{align*}
|Q_{\beta}[\varphi]|\leq&\int_{\Omega_{R}}\frac{C|\varphi(x',h(x'))||x'|}{(\varepsilon+|x'|^{m})^{2}}\,dx+C\|\varphi\|_{C^{2}(\partial D)}\\
\leq&C\eta\int_{|x'|<R}\frac{|x'|^{k+1}}{\varepsilon+|x'|^{m}}\,dx+C\|\varphi\|_{C^{2}(\partial D)}\\
\leq&C(\eta\rho_{k+1}(n,m;\varepsilon)+\|\varphi\|_{C^{2}(\partial D)}),
\end{align*}
Then (\ref{HND1116}) under ({\bf{\em A2}}) is proved.

{\bf Step 2.} Proof of (\ref{HND1116}) under ({\bf{\em A1}}). Take $\beta=n+1$ for instance. Denote
\begin{align*}
Q_{n+1}[\varphi]=&\sum^{n}_{j=1}\int_{\partial D_{1}}\frac{\partial u_{0j}}{\partial\nu_{0}}\Big|_{+}\cdot\psi_{n+1}=:\sum^{n}_{j=1}Q_{n+1,j}[\varphi],
\end{align*}
where $u_{0j}$, $j=1,2,\cdots,n$, are defined by (\ref{RTP101}). By definition,
\begin{align*}
Q_{n+1,j}[\varphi]=&\int_{\partial D_{1}}\left[\lambda\sum^{n}_{k=1}\partial_{k}u_{0j}^{k}\nu_{1}+\mu\sum^{n}_{i=1}(\partial_{i}u_{0j}^{1}+\partial_{1}u_{0j}^{i})\nu_{i}\right]x_{2}\\
&-\int_{\partial D_{1}}\left[\lambda\sum^{n}_{k=1}\partial_{k}u_{0j}^{k}\nu_{2}+\mu\sum^{n}_{i=1}(\partial_{i}u_{0j}^{2}+\partial_{2}u_{0j}^{i})\nu_{i}\right]x_{1}.
\end{align*}

{\bf Step 2.1.} If $n=2$, for $j=1$, we split $Q_{3,1}[\varphi]$ as follows:
\begin{align*}
Q_{3,1}^{1}[\varphi]=&\int_{\partial D_{1}}\left[\lambda\sum^{2}_{k=1}\partial_{k}u_{01}^{k}\nu_{1}+\mu\sum^{2}_{i=1}(\partial_{i}u_{01}^{1}+\partial_{1}u_{01}^{i})\nu_{i}\right]x_{2},\\
Q_{3,1}^{2}[\varphi]=&-\int_{\partial D_{1}}\Big(\lambda\partial_{1}u_{01}^{1}\nu_{2}+\mu(\partial_{1}u_{01}^{2}+\partial_{2}u_{01}^{1})\nu_{1}\Big)x_{1},
\end{align*}
and
\begin{align*}
Q_{3,1}^{3}[\varphi]=&-(\lambda+2\mu)\int_{\partial D_{1}}\partial_{2}u_{01}^{2}\nu_{2}x_{1}.
\end{align*}
Due to the fact that $|x_{2}|=|\varepsilon+h_{1}(x_{1})|\leq C(\varepsilon+|x_{1}|^{m})$ and $|\nu_{1}|\leq C|x_{1}|^{m-1}$ on $\Gamma^{+}_{R}$, it follows from Lemma \ref{lem89} that
\begin{align*}
|Q_{3,1}^{1}[\varphi]|\leq C(\eta+\|\varphi\|_{C^{2}(\partial D)}),\quad |Q_{3,1}^{2}[\varphi]|\leq C(\eta+\|\varphi\|_{C^{2}(\partial D)}),
\end{align*}
while, in light of $\tilde{u}_{01}^{2}=0$, we have
\begin{align*}
|Q_{3,1}^{3}[\varphi]|=&|\lambda+2\mu|\Big|\int_{\Gamma^{+}_{R}}\partial_{2}(u_{01}^{2}-\tilde{u}_{01}^{2})\nu_{2}x_{1}\Big|+C\|\varphi\|_{C^{2}(\partial D)}\leq  C(\eta+\|\varphi\|_{C^{2}(\partial D)}).
\end{align*}
Thus
\begin{align}\label{LCN1}
|Q_{3,1}[\varphi]|\leq&C(\eta+\|\varphi\|_{C^{2}(\partial D)}).
\end{align}
For $j=2$, $Q_{3,2}[\varphi]$ can be divided into two parts in the following.
\begin{align*}
Q_{3,2}^{1}[\varphi]=&\int_{\partial D_{1}}\left[\lambda\sum^{2}_{k=1}\partial_{k}u_{02}^{k}\nu_{1}+\mu\sum^{2}_{i=1}(\partial_{i}u_{02}^{1}+\partial_{1}u_{02}^{i})\nu_{i}\right]x_{2}\\
&-\int_{\partial D_{1}}\Big(\lambda\partial_{1}u_{02}^{1}\nu_{2}+\mu(\partial_{1}u_{02}^{2}+\partial_{2}u_{02}^{1})\nu_{1}\Big)x_{1},
\end{align*}
and
\begin{align*}
Q_{3,2}^{2}[\varphi]=&-(\lambda+2\mu)\int_{\partial D_{1}}\partial_{2}u_{02}^{2}\nu_{2}x_{1}.
\end{align*}
Similarly as before, making use of Lemma \ref{lem89}, we have
\begin{align*}
|Q_{3,2}^{1}[\varphi]|\leq C(\eta+\|\varphi\|_{C^{2}(\partial D)}).
\end{align*}
In view of ({\bf{\em A1}}), we obtain
\begin{align}\label{LVCZ1}
|Q_{3,2}^{2}[\varphi]|\leq&|\lambda+2\mu|\left|\int_{\Gamma^{+}_{R}}\partial_{2}\tilde{u}_{02}^{2}\nu_{2}x_{1}+\int_{\Gamma^{+}_{R}}\partial_{2}(u_{02}^{2}-\tilde{u}_{02}^{2})\nu_{2}x_{1}\right|+C\|\varphi\|_{C^{2}(\partial D)}\notag\\
\leq&|\lambda+2\mu|\left|\int_{|x_{1}|<R}\frac{\varphi^{2}(x_{1},h(x_{1}))x_{1}}{\varepsilon+h_{1}(x_{1})-h(x_{1})}\right|+C\|\varphi\|_{C^{2}(\partial D)}\notag\\
\leq&C\|\varphi\|_{C^{2}(\partial D)}.
\end{align}
Then
\begin{align*}
|Q_{3,2}[\varphi]|\leq C(\eta+\|\varphi\|_{C^{2}(\partial D)}).
\end{align*}
This, together with (\ref{LCN1}), yields that
\begin{align*}
|Q_{3}[\varphi]|\leq C(\eta+\|\varphi\|_{C^{2}(\partial D)}).
\end{align*}

{\bf Step 2.2.} If $n\geq3$, for $j=1$, $Q_{n+1,1}[\varphi]$ can be split as follows:
\begin{align*}
Q_{n+1,1}^{1}[\varphi]=&\int_{\partial D_{1}}\left[\lambda\sum^{n}_{k=1}\partial_{k}u_{01}^{k}\nu_{1}+\mu\sum^{n-1}_{i=1}(\partial_{i}u_{01}^{1}+\partial_{1}u_{01}^{i})\nu_{i}+\mu\partial_{1}u_{01}^{n}\nu_{n}\right]x_{2}\\
&-\int_{\partial D_{1}}\left[\lambda\sum^{n}_{k=1}\partial_{k}u_{01}^{k}\nu_{2}+\mu\sum^{n}_{i=1}(\partial_{i}u_{01}^{2}+\partial_{2}u_{01}^{i})\nu_{i}\right]x_{1},\\
Q_{n+1,1}^{2}[\varphi]=&\int_{\partial D_{1}}\mu\partial_{n}u_{01}^{1}\nu_{n}x_{2}.
\end{align*}
In view of the fact that $\tilde{u}_{01}^{2}=\tilde{u}_{01}^{n}=0$, it follows from Lemma (\ref{lem89}) that
\begin{align*}
|Q_{n+1,1}^{1}[\varphi]|\leq&C(\eta+\|\varphi\|_{C^{2}(\partial D)}).
\end{align*}
Combining ({\bf{\em H5}}), ({\bf{\em A1}}) and Lemma \ref{lem89}, we have
\begin{align}\label{PMNW1}
|Q_{n+1,1}^{2}[\varphi]|\leq&\mu\left|\int_{|x'|<R}\frac{\varphi^{1}(x',h(x'))x_{2}}{\varepsilon+h_{1}(x')-h(x')}\right|+C\|\varphi\|_{C^{2}(\partial D)}
\leq\,C\|\varphi\|_{C^{2}(\partial D)}.
\end{align}

By the same argument, we deduce that for $j=2,\cdots,n$, $|Q_{n+1,j}[\varphi]|\leq C(\eta+\|\varphi\|_{C^{2}(\partial D)})$. Then,
\begin{align*}
|Q_{n+1}[\varphi]|\leq&C(\eta+\|\varphi\|_{C^{2}(\partial D)}).
\end{align*}

{\bf Step 3.} Proof of (\ref{HND1112}) under ({\bf{\em A2}}). Take $\beta=1$ for example. The other cases are the same. Denote
\begin{align*}
Q_{1}[\varphi]=&\sum^{n}_{j=1}\int_{\partial D_{1}}\frac{\partial u_{0j}}{\partial\nu_{0}}\Big|_{+}\cdot\psi_{1}=:\sum^{n}_{j=1}Q_{1j}[\varphi],
\end{align*}
where $u_{0j}$, $j=1,2,\cdots,n$, are defined by (\ref{RTP101}). By definition,
\begin{align*}
Q_{11}[\varphi]=&\int_{\partial D_{1}}\left[\lambda\sum^{n}_{k=1}\partial_{k}u_{01}^{k}\nu_{1}+\mu\sum^{n}_{i=1}(\partial_{1}u_{01}^{i}+\partial_{i}u_{01}^{1})\nu_{i}\right].
\end{align*}
To estimate $Q_{11}[\varphi]$, we first split it into two parts in the following.
\begin{align*}
Q_{11}^{1}[\varphi]=&\int_{\partial D_{1}}\left[\lambda\sum^{n}_{k=1}\partial_{k}u_{01}^{k}\nu_{1}+\mu\sum^{n-1}_{i=1}(\partial_{1}u_{01}^{i}+\partial_{i}u_{01}^{1})\nu_{i}+\mu\partial_{1}u_{01}^{n}\nu_{n}\right],\\
Q_{11}^{2}[\varphi]=&\mu\int_{\partial D_{1}}\partial_{n}u_{01}^{1}\nu_{n}.
\end{align*}
Utilizing Lemma \ref{lem89}, we derive
\begin{align}\label{KKQ1}
|Q_{11}^{1}[\varphi]|\leq\int_{\partial D_{1}\cap\Gamma^{+}_{R}}C\eta|x'|^{k-1}+C\|\varphi\|_{C^{2}(\partial D)}\leq C(\eta+\|\varphi\|_{C^{2}(\partial D)}).
\end{align}

As for $Q_{11}^{2}[\varphi]$, we can split it as follows:
\begin{align*}
Q_{11}^{2}[\varphi]=&\mu\int_{\partial D_{1}\cap\Gamma^{+}_{R}}\partial_{n}(u_{01}^{1}-\tilde{u}_{01}^{1})\nu_{n}+\mu\int_{\partial D_{1}\cap\Gamma^{+}_{R}}\partial_{n}\tilde{u}_{01}^{1}\nu_{n}+O(1)\|\varphi\|_{C^{2}(\partial D)},\\
=&:Q_{11}^{2,1}[\varphi]+Q_{11}^{2,2}[\varphi]+O(1)\|\varphi\|_{C^{2}(\partial D)}.
\end{align*}
It follows from (\ref{RTP1311}) that
\begin{align}\label{KKQ2}
|Q_{11}^{2,1}[\varphi]|\leq\int_{\partial D_{1}\cap\Gamma^{+}_{R}}C\eta|x'|^{k-1}\leq C\eta.
\end{align}
Making use of ({\bf{\em H5}}) and ({\bf{\em A2}}), we obtain
\begin{align}\label{PMNW3}
|Q_{11}^{2,2}[\varphi]|=&\mu\left|\int_{|x'|<R}\frac{\varphi^{1}(x',h(x'))}{\varepsilon+h_{1}(x')-h(x')}\right|=0.
\end{align}

Next, due to the fact that $\tilde{u}_{0j}^{1}=0,$ $j=2,\cdots,n$, we have
\begin{align*}
Q_{1j}[\varphi]=&\int_{\partial D_{1}}\left[\lambda\sum^{n}_{k=1}\partial_{k}u_{0j}^{k}\nu_{1}+\mu\sum^{n}_{i=1}(\partial_{1}u_{0j}^{i}+\partial_{i}u_{0j}^{1})\nu_{i}\right]\\
=&\int_{\partial D_{1}}\Bigg[\lambda\sum^{n}_{k=1}\partial_{k}u_{0j}^{k}\nu_{1}+\mu\sum^{n-1}_{i=1}(\partial_{1}u_{0j}^{i}+\partial_{i}u_{0j}^{1})\nu_{i}\\
&\qquad\quad+\mu\partial_{1}u_{0j}^{n}\nu_{n}+\mu\partial_{n}(u_{0j}^{1}-\tilde{u}_{0j}^{1})\nu_{n}\Bigg].
\end{align*}
Making use of Lemma \ref{lem89} again, we derive that for $j=2,\cdots,n$,
$$|Q_{1j}[\varphi]|\leq C(\eta+\|\varphi\|_{C^{2}(\partial D)}).$$
This, together with (\ref{KKQ1})--(\ref{PMNW3}), leads to $|Q_{1}[\varphi]|\leq C(\eta+\|\varphi\|_{C^{2}(\partial D)})$.

{\bf Step 4.} Proof of (\ref{HND1112}) and (\ref{HND1116}) under ({\bf{\em A3}}). Due to assumptions ({\bf{\em H5}}) and ({\bf{\em A3}}), we deduce that for $i,j=1,\cdots,n-1$, $i\neq j$, $\varphi^{i}(x',h(x'))x_{j}$ is odd with respect to $x_{i}$, and $\varphi^{n}(x',h(x'))x_{j}$ is odd with respect to $x_{1}$ or $x_{2}$. Then (\ref{LVCZ1}), (\ref{PMNW1}) and (\ref{PMNW3}) also hold. We thus obtain
\begin{align*}
|Q_{1}[\varphi]|\leq&C(\eta+\|\varphi\|_{C^{2}(\partial D)}),\quad|Q_{n+1}[\varphi]|\leq C(\eta+\|\varphi\|_{C^{2}(\partial D)}).
\end{align*}
By the same argument, we derive that $|Q_{\beta}[\varphi]|\leq C(\eta+\|\varphi\|_{C^{2}(\partial D)})$ for $\beta=1,2,\cdots,\frac{n(n+1)}{2}$.

\section{Proof of Theorem \ref{MMM896}}
Similarly as before, we can obtain

\begin{lemma}\label{KM323}
Assume as in Theorem \ref{MMM896}. Then,

$(a)$ If $(\mathbf{\Phi1})$ or $(\mathbf{\Phi2})$ holds, for $\beta=1,2,\cdots,n$,
\begin{align}\label{LLCCAA1}
\frac{\eta\rho_{k}(n,m;\varepsilon)}{C}-C\|\varphi\|_{C^{2}(\partial D)}\leq|Q_{\beta}[\varphi]|\leq C(\eta\rho_{k}(n,m;\varepsilon)+\|\varphi\|_{C^{2}(\partial D)}),
\end{align}
and, for $\beta=n+1,\cdots,\frac{n(n+1)}{2}$, if $(\mathbf{\Phi1})$ holds,
\begin{align}\label{LLCCAA2}
|Q_{\beta}[\varphi]|\leq C(\eta+\|\varphi\|_{C^{2}(\partial D)}),
\end{align}
while, if $(\mathbf{\Phi2})$ holds,
\begin{align}\label{Lllk2}
|Q_{\beta}[\varphi]|\leq C(\eta\rho_{2}(n,m;\varepsilon)+\|\varphi\|_{C^{2}(\partial D)});
\end{align}

$(b)$ If $(\mathbf{\Phi3})$, $(\mathbf{\Phi4})$ or $(\mathbf{\Phi5})$ holds, for $\beta=1,2,\cdots,\frac{n(n+1)}{2}$,
\begin{align}\label{AAPP321}
|Q_{\beta}[\varphi]|\leq C(\eta+\|\varphi\|_{C^{2}(\partial D)}).
\end{align}

\end{lemma}

\begin{proof}
{\bf Step 1.} Proofs of (\ref{LLCCAA1}), (\ref{LLCCAA2}) and (\ref{Lllk2}). We only prove (\ref{LLCCAA1}) under $(\mathbf{\Phi1})$. The proof of (\ref{LLCCAA1}) under $(\mathbf{\Phi2})$ is similar and thus omitted. To estimate $|Q_{\beta}[\varphi]|$ for $\beta=1,2,\cdots,n$, we take $\beta=1$ for example. Similarly as in the proof of Proposition \ref{lekm323}, we denote
\begin{align*}
Q_{1}[\varphi]=&\sum^{n}_{j=1}Q_{1j}[\varphi],
\end{align*}
where
\begin{align*}
Q_{1j}[\varphi]=&\int_{\partial D_{1}}\left[\lambda\sum^{n}_{k=1}\partial_{k}u_{0j}^{k}\nu_{1}+\mu\sum^{n}_{i=1}(\partial_{1}u_{0j}^{i}+\partial_{i}u_{0j}^{1})\nu_{i}\right].
\end{align*}
For $j=1$, $Q_{11}[\varphi]$ can be split in the following.
\begin{align*}
Q_{11}^{1}[\varphi]=&\int_{\partial D_{1}}\Bigg[\lambda\sum^{n}_{k=1}\partial_{k}u_{01}^{k}\nu_{1}+\mu\sum^{n-1}_{i=1}(\partial_{1}u_{01}^{i}+\partial_{i}u_{01}^{1})\nu_{i}+\mu\partial_{1}u_{01}^{n}\nu_{n}\\
&\qquad\quad+\mu\partial_{n}(u_{01}^{1}-\tilde{u}_{01}^{1})\nu_{n}\Bigg],\\
Q_{11}^{2}[\varphi]=&\,\mu\int_{\partial D_{1}}\partial_{n}\tilde{u}_{01}^{1}\nu_{n}.
\end{align*}
By using Lemma \ref{lem89}, we have
\begin{align}\label{YTQ1}
|Q_{11}^{1}[\varphi]|\leq\int_{\partial D_{1}\cap\Gamma^{+}_{R}}C\eta|x'|^{k-1}+C\|\varphi\|_{C^{2}(\partial D)}\leq C(\eta+\|\varphi\|_{C^{2}(\partial D)}).
\end{align}

As for $Q_{11}^{2}[\varphi]$, on one hand,
\begin{align}\label{YTQ3}
|Q_{11}^{2}[\varphi]|\leq&\left|\mu\int_{\partial D_{1}\cap\Gamma^{+}_{R}}\frac{\varphi^{1}(x',h(x'))}{\varepsilon+h_{1}(x')-h(x')}\nu_{n}\right|+C\|\varphi\|_{C^{2}(\partial D)}\notag\\
\leq&\mu\int_{|x'|<R}\frac{\eta|x'|^{k}}{\varepsilon+|x'|^{m}}dx'+C\|\varphi\|_{C^{2}(\partial D)}\notag\\
\leq&C(\eta\rho_{k}(n,m;\varepsilon)+\|\varphi\|_{C^{2}(\partial D)}).
\end{align}
On the other hand,
\begin{align}\label{YTQ5}
|Q_{11}^{2}[\varphi]|\geq&\left|\mu\int_{\partial D_{1}\cap\Gamma^{+}_{R}}\frac{\varphi^{1}(x',h(x'))}{\varepsilon+h_{1}(x')-h(x')}\nu_{n}\right|-C\|\varphi\|_{C^{2}(\partial D)}\notag\\
\geq&\mu\int_{|x'|<R}\frac{\eta|x'|^{k}}{\varepsilon+|x'|^{m}}dx'-C\|\varphi\|_{C^{2}(\partial D)}\notag\\
\geq&\frac{\eta\rho_{k}(n,m;\varepsilon)}{C}-C\|\varphi\|_{C^{2}(\partial D)}.
\end{align}
Thus, combining with (\ref{YTQ1})--(\ref{YTQ5}), we have
\begin{align}\label{HFR65}
\frac{\eta\rho_{k}(n,m;\varepsilon)}{C}-C\|\varphi\|_{C^{2}(\partial D)}\leq|Q_{11}[\varphi]|\leq C(\eta\rho_{k}(n,m;\varepsilon)+\|\varphi\|_{C^{2}(\partial D)}),
\end{align}

For $Q_{1j}[\varphi]$, $j=2,\cdots,n$,  by the same argument as before, in view of the fact that $\tilde{u}_{0j}^{1}=0$, it follows from Lemma \ref{lem89} that
$$|Q_{1j}[\varphi]|\leq C(\eta+\|\varphi\|_{C^{2}(\partial D)}).$$
This, together with (\ref{HFR65}), yields that (\ref{LLCCAA1}) holds.

(\ref{LLCCAA2}) and (\ref{Lllk2}) can be proved by the same argument as {\bf Step 1} and {\bf Step 2} in the proof of Proposition \ref{lekm323} and thus their proofs are omitted here.

{\bf Step 2.} Proof of (\ref{AAPP321}). Under condition $(\mathbf{\Phi3})$, $(\mathbf{\Phi4})$ or $(\mathbf{\Phi5})$, we deduce that for $\beta=1,\cdots,\frac{n(n+1)}{2}$,
\begin{align*}
|Q_{\beta}[\varphi]|\leq&\int_{|x'|<R}\frac{C|\varphi(x',h(x'))|}{\varepsilon+|x'|^{m}}\,dx+C\|\varphi\|_{C^{2}(\partial D)}
\leq\,C(\eta+\|\varphi\|_{C^{2}(\partial D)}).
\end{align*}
Then (\ref{AAPP321}) holds.

\end{proof}

We next estimate $|Q_{\beta}[\varphi]-Q^{\ast}_{\beta}[\varphi]|$, $\beta=1,2,\cdots,\frac{n(n+1)}{2}$ under condition $(\mathbf{\Phi3})$, $(\mathbf{\Phi4})$ or $(\mathbf{\Phi5})$.

\begin{lemma}\label{lemmaLw.01}
Assume $m\leq n$. Then, for $\beta=1,2,\cdots,n$,
\begin{align*}
|Q_{\beta}[\varphi]-Q^{\ast}_{\beta}[\varphi]|\leq C&
\begin{cases}
\varepsilon^{\frac{k-1}{(m+k-1)(m+1)}},&m=n,\;k>1\\
\varepsilon^{\frac{n-m}{n(m+1)}},&m<n.
\end{cases}
\end{align*}
Consequently,
\begin{align*}
Q_{\beta}[\varphi]\rightarrow Q^{\ast}_{\beta}[\varphi],\;\;\, \mathrm{as}\;\, \varepsilon\rightarrow0,\;\,\,\mathrm{for}\;\,\beta=1,2,\cdots,n.
\end{align*}

\end{lemma}

\begin{proof}
Take the case $\beta=1$ for example. The other cases are the same. Recalling the definition of $u_{0}$ and $u_{1}$ and utilizing (\ref{Le2.01222}), we deduce that
\begin{align*}
Q_{1}[\varphi]=\int_{\partial D}\frac{\partial u_{1}}{\partial\nu_{0}}\Big|_{+}\cdot\varphi(x).
\end{align*}
Likewise,
\begin{align*}
Q^{\ast}_{1}[\varphi]=\int_{\partial D}\frac{\partial u^{\ast}_{1}}{\partial\nu_{0}}\Big|_{+}\cdot\varphi(x),
\end{align*}
where $u_{1}^{\ast}$ solves
\begin{equation}\label{l03.001}
\begin{cases}
\mathcal{L}_{\lambda,\mu}u_{1}^{\ast}=0,\quad\;\,&\mathrm{in}\;\,\Omega^{\ast},\\
u_{1}^{\ast}=\psi_{1},\quad\;\,&\mathrm{on}\;\,\partial D_{1}^{\ast}\setminus\{0\},\\
u_{1}^{\ast}=0,\quad\;\,&\mathrm{on}\;\,\partial D.
\end{cases}
\end{equation}
Therefore,
\begin{align*}
Q_{1}[\varphi]-Q^{\ast}_{1}[\varphi]=\int_{\partial D}\frac{\partial(u_{1}-u_{1}^{\ast})}{\partial\nu_{0}}\Big|_{+}\cdot\varphi(x).
\end{align*}

To estimate $u_{1}-u_{1}^{\ast}$, we first introduce two auxiliary functions
$$\tilde{u}_{1}=(\bar{v},0,\cdots,0),\quad\mathrm{and}\quad\tilde{u}_{1}^{\ast}=(\bar{v}^{\ast},0,\cdots,0),$$
where $\bar{v}$ is defined in Section 2, and $\bar{v}^{\ast}$ satisfies $\bar{v}^{\ast}=1$ on $\partial D_{1}^{\ast}\setminus\{0\}$, $\bar{v}^{\ast}=0$ on $\partial D$, and
$$\bar{v}^{\ast}=\frac{x_{n}-h(x')}{h_{1}(x')-h(x')},\quad\mathrm{in}\;\,\Omega_{2R}^{\ast},\quad\;\,\|\bar{v}^{\ast}\|_{C^{2}(\Omega^{\ast}\setminus\Omega_{R}^{\ast})}\leq C,$$
where $\Omega^{\ast}_{r}:=\Omega^{\ast}\cap\{|x'|<r\},$ $0<r\leq2R$. In light of ({\bf{\em H3}}), we derive that for $x\in\Omega_{R}^{\ast}$,
\begin{align}\label{Lw3.003}
|\nabla_{x'}(\tilde{u}^{1}_{1}-\tilde{u}^{\ast1}_{1})|\leq\frac{C}{|x'|},
\end{align}
and
\begin{align}\label{Lw3.004}
|\partial_{n}(\tilde{u}^{1}_{1}-\tilde{u}^{\ast1}_{1})|\leq\frac{C\varepsilon}{|x'|^{m}(\varepsilon+|x'|^{m})}.
\end{align}
Applying Corollary \ref{co.001} to (\ref{l03.001}), it follows that for $x\in\Omega_{R}^{\ast}$,
\begin{align}\label{Lw3.005}
|\nabla(u_{1}^{\ast}-\tilde{u}_{1}^{\ast})|\leq\frac{C}{|x'|},
\end{align}
and
\begin{align}\label{Lw3.006}
|\nabla_{x'}u_{1}^{\ast}|\leq\frac{C}{|x'|},\quad\;|\partial_{n}u_{1}^{\ast}|\leq\frac{C}{|x'|^{m}}.
\end{align}
For $0<r<R$, denote
\begin{align}\label{LNM656}
\mathcal{C}_{r}:=\left\{x\in\mathbb{R}^{n}\Big|\;|x'|<r,\,\frac{1}{2}\min_{|x'|\leq r}h(x')\leq x_{n}\leq\varepsilon+2\max_{|x'|\leq r}h_{1}(x')\right\}.
\end{align}
We now divide into two steps to estimate $|Q_{1}[\varphi]-Q^{\ast}_{1}[\varphi]|$.

{\bf STEP 1.} Note that $u_{1}-u_{1}^{\ast}$ solves
\begin{align*}
\begin{cases}
\mathcal{L}_{\lambda,\mu}(u_{1}-u_{1}^{\ast})=0,&\mathrm{in}\;\,D\setminus(\overline{D_{1}\cup D_{1}^{\ast}}),\\
u_{1}-u_{1}^{\ast}=\psi_{1}-u_{1}^{\ast},&\mathrm{on}\;\,\partial D_{1}\setminus D_{1}^{\ast},\\
u_{1}-u_{1}^{\ast}=u_{1}-\psi_{1},&\mathrm{on}\;\,\partial D_{1}^{\ast}\setminus(D_{1}\cup\{0\}),\\
u_{1}-u_{1}^{\ast}=0,&\mathrm{on}\;\,\partial D.
\end{cases}
\end{align*}
To start with, we estimate $|u_{1}-u_{1}^{\ast}|$ on $\partial(D_{1}\cup D_{1}^{\ast})\setminus\mathcal{C}_{\varepsilon^{\gamma}}$, where $0<\gamma<1/2$ to be determined later. Recalling the definition of $u_{1}^{\ast}$, we know that
$$|\partial_{n}u_{1}^{\ast}|\leq C,\quad\;\,\mathrm{in}\;\,\Omega^{\ast}\setminus\Omega^{\ast}_{R}.$$
Thus,
\begin{align}\label{Lw3.007}
|u_{1}-u_{1}^{\ast}|\leq C\varepsilon,\quad\;\,\mathrm{for}\;\,x\in\partial D_{1}\setminus D_{1}^{\ast}.
\end{align}
And making use of (\ref{Le2.029}), we have
\begin{align}\label{Lw3.008}
|u_{1}-u_{1}^{\ast}|\leq C\varepsilon^{1-m\gamma},\quad\;\,\mathrm{on}\;\,\partial D_{1}^{\ast}\setminus(D_{1}\cup\mathcal{C}_{\varepsilon^{\gamma}}).
\end{align}
It follows from (\ref{Le2.0028}), (\ref{Lw3.004}) and (\ref{Lw3.005}) that for $x\in\Omega_{R}^{\ast}\cap\{|x'|=\varepsilon^{\gamma}\}$,
\begin{align*}
|\partial_{n}(u_{1}-u_{1}^{\ast})|\leq&|\partial_{n}(u_{1}-\tilde{u}_{1})|+|\partial_{n}(\tilde{u}_{1}-\tilde{u}_{1}^{\ast})|+|\partial_{n}(u_{1}^{\ast}-\tilde{u}_{1}^{\ast})|\\
\leq&C\left(\frac{1}{\varepsilon^{2m\gamma-1}}+\frac{1}{\varepsilon^{\gamma}}\right).
\end{align*}
Then, in view of the fact that $\tilde{u}_{1}-\tilde{u}_{1}^{\ast}=0$ on $\partial D$, we obtain
\begin{align}
|(u_{1}-u_{1}^{\ast})(x',x_{n})|=&|(u_{1}-u_{1}^{\ast})(x',x_{n})-(u_{1}-u_{1}^{\ast})(x',h(x'))|\notag\\
\leq&C\big(\varepsilon^{1-m\gamma}+\varepsilon^{(m-1)\gamma}\big).\label{Lw3.009}
\end{align}
Take $\gamma=\frac{1}{m+1}$. Then combining with (\ref{Lw3.007}), (\ref{Lw3.008}) and (\ref{Lw3.009}), we obtain
$$|u_{1}-u_{1}^{\ast}|\leq C\varepsilon^{\frac{1}{m+1}},\quad\;\,\mathrm{on}\;\,\partial\big(D\setminus\big(\overline{D_{1}\cup D_{1}^{\ast}\cup\mathcal{C}_{\varepsilon^{\frac{1}{m+1}}}}\big)\big).$$
Utilizing the maximum principle for Lam\'{e} system in \cite{MMN2007}, we deduce
$$|u_{1}-u_{1}^{\ast}|\leq C\varepsilon^{\frac{1}{m+1}},\quad\;\,\mathrm{on}\;\,D\setminus\big(\overline{D_{1}\cup D_{1}^{\ast}\cup\mathcal{C}_{\varepsilon^{\frac{1}{m+1}}}}\big).$$
Thus it follows from the standard interior and boundary estimates for Lam\'{e} system that, for any $\frac{m-1}{m(m+1)}<\tilde{\gamma}<\frac{1}{m+1}$,
$$|\nabla(u_{1}-u_{1}^{\ast})|\leq C\varepsilon^{m\tilde{\gamma}-\frac{m-1}{m+1}},\quad\;\,\mathrm{on}\;\,D\setminus\big(\overline{D_{1}\cup D_{1}^{\ast}\cup\mathcal{C}_{\varepsilon^{\frac{1}{m+1}-\tilde{\gamma}}}}\big),$$
which indicates that
\begin{align}\label{Lw3.010}
|\mathcal{B}^{out}|:=\left|\int_{\partial D\setminus\mathcal{C}_{\varepsilon^{\frac{1}{m+1}-\tilde{\gamma}}}}\frac{\partial(u_{1}-u_{1}^{\ast})}{\partial\nu_{0}}\Big|_{+}\cdot\varphi(x)\right|\leq C\varepsilon^{m\tilde{\gamma}-\frac{m-1}{m+1}},
\end{align}
where $\frac{m-1}{m(m+1)}<\tilde{\gamma}<\frac{1}{m+1}$ will be determined later.

{\bf STEP 2.} we next estimate
\begin{align*}
\mathcal{B}^{in}:=&\int_{\partial D\cap\mathcal{C}_{\varepsilon^{\frac{1}{m+1}-\tilde{\gamma}}}}\frac{\partial(u_{1}-u_{1}^{\ast})}{\partial\nu_{0}}\Big|_{+}\cdot\varphi(x)\\
=&\int_{\partial D\cap\mathcal{C}_{\varepsilon^{\frac{1}{m+1}-\tilde{\gamma}}}}\frac{\partial(\tilde{u}_{1}-\tilde{u}_{1}^{\ast})}{\partial\nu_{0}}\Big|_{+}\cdot\varphi(x)+\int_{\partial D\cap\mathcal{C}_{\varepsilon^{\frac{1}{m+1}-\tilde{\gamma}}}}\frac{\partial(w_{1}-\tilde{w}_{1}^{\ast})}{\partial\nu_{0}}\Big|_{+}\cdot\varphi(x)\\
=&:\mathcal{B}_{\tilde{u}}+\mathcal{B}_{\tilde{w}},
\end{align*}
where $w_{1}=u_{1}-\tilde{u}_{1}$, $w_{1}^{\ast}=u_{1}^{\ast}-\tilde{u}_{1}^{\ast}$. First,
\begin{align*}
\mathcal{B}_{\tilde{u}}=&\int_{\partial D\cap\mathcal{C}_{\varepsilon^{\frac{1}{m+1}-\tilde{\gamma}}}}\bigg\{\lambda\sum^{n}_{i=1}\partial_{1}(\tilde{u}^{1}_{1}-\tilde{u}^{\ast1}_{1})\nu_{i}\varphi^{i}(x)+\\
&\qquad+\mu\sum^{n-1}_{i=1}\partial_{i}(\tilde{u}^{1}_{1}-\tilde{u}^{\ast1}_{1})[\nu_{1}\varphi^{i}(x)+\nu_{i}\varphi^{1}(x)]+\mu\partial_{n}(\tilde{u}^{1}_{1}-\tilde{u}^{\ast1}_{1})\nu_{1}\varphi^{n}(x)\bigg\}\\
&+\int_{\partial D\cap\mathcal{C}_{\varepsilon^{\frac{1}{m+1}-\tilde{\gamma}}}}\mu\partial_{n}(\tilde{u}^{1}_{1}-\tilde{u}^{\ast1}_{1})\nu_{n}\varphi^{1}(x)\\
=&:\mathcal{B}^{1}_{\tilde{u}}+\mathcal{B}^{2}_{\tilde{u}}.
\end{align*}
According to (\ref{Lw3.003}), (\ref{Lw3.004}) and the Taylor expansion of $\varphi^{i}(x)$, we deduce that
\begin{align}\label{Lw3.011}
|\mathcal{B}^{1}_{\tilde{u}}|\leq C\varepsilon^{\left(\frac{1}{m+1}-\tilde{\gamma}\right)(n+k-2)},
\end{align}
and, for $m=n$ and $k>1$, we derive
\begin{align}\label{Lw3.012}
|\mathcal{B}^{2}_{\tilde{u}}|\leq C\varepsilon^{\left(\frac{1}{m+1}-\tilde{\gamma}\right)(k-1)};
\end{align}
for $m<n$,
\begin{align}\label{Lw3.013}
|\mathcal{B}^{2}_{\tilde{u}}|\leq C|\nabla_{x'}\varphi(0)|\varepsilon^{\left(\frac{1}{m+1}-\tilde{\gamma}\right)(n-m)}+C\|\nabla^{2}\varphi\|_{L^{\infty}(\partial D)}\varepsilon^{\left(\frac{1}{m+1}-\tilde{\gamma}\right)(n-m+1)}.
\end{align}
Then it follows from (\ref{Lw3.011})--(\ref{Lw3.013}) that if $m=n$ and $k>1$,
\begin{align}\label{Lw3.014}
|\mathcal{B}_{\tilde{u}}|\leq C\varepsilon^{\left(\frac{1}{m+1}-\tilde{\gamma}\right)(k-1)};
\end{align}
if $m<n$,
\begin{align}\label{Lw3.015}
|\mathcal{B}_{\tilde{u}}|\leq C\big(|\nabla_{x'}\varphi(0)|+\|\nabla^{2}\varphi\|_{L^{\infty}(\partial D)}\big)\varepsilon^{\left(\frac{1}{m+1}-\tilde{\gamma}\right)(n-m)}.
\end{align}

Next, we estimate $\mathcal{B}_{w}$. Using Corollary \ref{co.001}, we obtain that for $0<|x'|\leq R$,
\begin{align}\label{Lw3.016}
|\nabla w_{1}|\leq\frac{C}{\sqrt[m]{\varepsilon+|x'|^{m}}},\quad\;\,|\nabla w_{1}^{\ast}|\leq\frac{C}{|x'|}.
\end{align}
By definition,
\begin{align*}
\mathcal{B}_{w}=&\int_{\partial D\cap\mathcal{C}_{\varepsilon^{\frac{1}{m+1}-\tilde{\gamma}}}}\bigg\{\lambda\sum^{n}_{i,j=1}\partial_{i}(w_{1}^{i}-w_{1}^{\ast i})\nu_{j}\varphi^{j}(x)\\
&\quad+\mu\sum^{n}_{i,j=1}[\partial_{j}(w_{1}^{i}-w_{1}^{\ast i})+\partial_{i}(w_{1}^{j}-w_{1}^{\ast j})]\nu_{j}\varphi^{i}(x)\bigg\}.
\end{align*}
It follows from (\ref{Lw3.016}) and the Taylor expansion of $\varphi^{j}(x)$ that
\begin{align}\label{Lw3.017}
|\mathcal{B}_{w}|\leq C\big(|\nabla_{x'}\varphi(0)|+\|\nabla^{2}\varphi\|_{L^{\infty}(\partial D)}\big)\varepsilon^{\left(\frac{1}{m+1}-\tilde{\gamma}\right)(n-1)}.
\end{align}
Utilizing (\ref{Lw3.010}), (\ref{Lw3.014}) and (\ref{Lw3.017}) and picking $\tilde{\gamma}=\frac{m+k-2}{(m+k-1)(m+1)}$, we deduce that for $m=n$ and $k>1$,
\begin{align*}
|Q_{1}[\varphi]-Q^{\ast}_{1}[\varphi]|\leq C\varepsilon^{\frac{k-1}{(m+k-1)(m+1)}}.
\end{align*}
For $m<n$, taking $\tilde{\gamma}=\frac{n-1}{n(m+1)}$, it follows from (\ref{Lw3.010}), (\ref{Lw3.015}) and (\ref{Lw3.017}) that
\begin{align*}
|Q_{1}[\varphi]-Q^{\ast}_{1}[\varphi]|\leq C\varepsilon^{\frac{n-m}{n(m+1)}}.
\end{align*}
Thus, Lemma \ref{lemmaLw.01} is established.
\end{proof}

\subsection{Proof of Theorem \ref{MMM896}}
In view of assumptions $(\mathbf{\Phi1})$-$(\mathbf{\Phi5})$ in Theorem \ref{MMM896}, it follows from Lemma \ref{KM323} and Lemma \ref{lemmaLw.01} that there exists a universal constant $C_{0}>0$ such that for a sufficiently small $\varepsilon>0$,
\begin{align}\label{Lw3.018}
|Q_{k_{0}}[\varphi]|>\frac{\rho_{k}(n,m;\varepsilon)}{C_{0}}.
\end{align}
Denote $A^{\ast}=(a^{\ast}_{\alpha\beta})_{n\times n}$, where $a^{\ast}_{\alpha\beta}$ is the cofactor of $a_{\alpha\beta}$. Utilizing Lemma \ref{lemmabc}, we obtain
\begin{align}\label{Lw3.019}
a^{\ast}_{\alpha\alpha}\thicksim(\rho_{0}(n,m;\varepsilon))^{n-1},\;\,\alpha=1,2,\cdots,n;\;\,\;a^{\ast}_{\alpha\beta}\thicksim(\rho_{0}(n,m;\varepsilon))^{n-2},\;\,\alpha\neq\beta.
\end{align}

Recalling the definition of $C^{k_{0}}$ and utilizing Lemma \ref{lemmabc}, Lemma \ref{KM323}, (\ref{GMD1}) and (\ref{Lw3.018})--(\ref{Lw3.019}), we obtain that for a sufficiently small $\varepsilon>0$,
\begin{align}\label{Lw3.022}
|C^{k_{0}}|\geq\frac{1}{2\det A}a^{\ast}_{k_{0}k_{0}}|Q_{k_{0}}[\varphi]|\geq\frac{\rho_{k}(n,m;\varepsilon)}{C\rho_{0}(n,m;\varepsilon)}.
\end{align}
Here we used the assumption that $Q^{\ast}_{\beta}[\varphi]=0$ for all $\beta\neq k_{0}$, $1\leq\beta\leq n$ when $m+1<n$.

Due to Corollary \ref{co.001}, we deduce that for $x\in\Omega_{R}$,
\begin{align}\label{Lw3.023}
|\partial_{n}u^{k_{0}}_{k_{0}}|\geq|\partial_{n}\tilde{u}^{k_{0}}_{k_{0}}|-|\partial_{n}(u^{k_{0}}_{k_{0}}-\tilde{u}^{k_{0}}_{k_{0}})|\geq\frac{1}{C(\varepsilon+|x'|^{m})},
\end{align}
and in light of $\tilde{u}^{k_{0}}_{\alpha}$ for $\alpha\neq k_{0}$,
\begin{align}\label{Lw3.025}
|\partial_{n}u^{k_{0}}_{\alpha}|=|\partial_{n}(u^{k_{0}}_{\alpha}-\tilde{u}^{k_{0}}_{\alpha})|\leq\frac{C}{\sqrt[m]{\varepsilon+|x'|^{m}}}.
\end{align}
On one hand, it follows from (\ref{Lw3.022})--(\ref{Lw3.025}) that for $x\in\Omega_{R}\cap\{x'=0'\}$,
\begin{align}\label{Lw3.026}
\left|\sum^{n}_{\alpha=1}C^{\alpha}\nabla u_{\alpha}\right|\geq\left|\sum^{n}_{\alpha=1}C^{\alpha}\partial_{n}u_{\alpha}^{k_{0}}\right|\geq\frac{\rho_{k}(n,m;\varepsilon)}{C\varepsilon\rho_{0}(n,m;\varepsilon)}.
\end{align}

On the other hand, combining with Corollary \ref{co.001}, Lemma \ref{lem89} and Lemma \ref{KM323}, we derive that for $x\in\Omega_{R}\cap\{x'=0'\}$,
\begin{align}\label{Lw3.027}
\left|\sum^{\frac{n(n+1)}{2}}_{\alpha=n+1}C^{\alpha}\nabla u_{\alpha}\right|\leq C,\quad\;\,|\nabla u_{0}|\leq C.
\end{align}
Consequently, making use of (\ref{Lw3.026})--(\ref{Lw3.027}) and (\ref{Le2.015}), we obtain
\begin{align*}
|\nabla u(0',x_{n})|\geq\frac{\rho_{k}(n,m;\varepsilon)}{C\varepsilon\rho_{0}(n,m;\varepsilon)},\quad\;\,0<x_{n}<\varepsilon.
\end{align*}
Then, (\ref{PPN1}) is proved.

%

\section{The lower bound on the cylinder surface $\{|x'|=\sqrt[m]{\varepsilon}\}\cap\Omega$}
Similarly as before, to establish the lower bound on the cylinder surface $\{|x'|=\sqrt[m]{\varepsilon}\}\cap\Omega$, we assume that $m>n$,
{\it
\begin{enumerate}
\item[$(\mathbf{\widetilde{\Phi}1})$]
for $k>m-n$, if (\ref{zw11}) holds and $Q^{\ast}_{2n-1}[\varphi]\neq0$;
\item[$(\mathbf{\widetilde{\Phi}2})$]
for $k=m-n$ and $m>n+1$, if (\ref{KHMDC1}) holds and
\begin{align}\label{KKT66}
\varphi^{i}(x)=\eta x_{1}|x_{1}|^{k-1},\quad\;\,x\in\Gamma^{-}_{R},\;\mathrm{for}\;i=1,2,\cdots,n,
\end{align}
and we additionally assume $\lambda+2\mu\neq0$;
\item[$(\mathbf{\widetilde{\Phi}3})$]
for $k=1$, $m=n+1$ or for $k<m-n$, if (\ref{KHMDC1}) holds, and
\begin{align}\label{PDSQ131}
\varphi^{i}(x)=\eta x_{i}|x_{i}|^{k-1},\; \varphi^{n}(x)=0,\quad\;\,x\in\Gamma^{-}_{R},\;i=1,\cdots,n-1,
\end{align}
\end{enumerate}
where $\eta$ is a positive constant independent of $\varepsilon$.}

\begin{theorem}\label{AAA666666}
Assume that $D_{1}\subset D\subseteq\mathbb{R}^{n}\,(n\geq2)$ are defined as above and conditions ({\bf{\em H1}})--({\bf{\em H5}}) hold with $\Sigma'=\{0'\}$. Let $u\in H^{1}(D;\mathbb{R}^{n})\cap C^{1}(\overline{\Omega};\mathbb{R}^{n})$ be a solution of (\ref{La.002}). Assume $m>n$. Then for a sufficiently small $\varepsilon>0$, we obtain that if $(\mathbf{\widetilde{\Phi}1})$ or $(\mathbf{\widetilde{\Phi}2})$ holds, then
\begin{align}\label{PPN2}
|\nabla u(x)|\geq\frac{\eta\rho_{k+1;2}(n,m;\varepsilon)}{C\varepsilon^{1-\frac{1}{m}}},\quad\mathrm{for}\;\,x\in\{x'=(\sqrt[m]{\varepsilon},0,\cdots,0)\}\cap\Omega,
\end{align}
while, if $(\mathbf{\widetilde{\Phi}3})$ holds,
\begin{align}\label{PPNM6}
|\nabla u(x)|\geq\frac{\eta}{C\varepsilon^{1-\frac{k}{m}}},\quad\mathrm{for}\;\,x\in\{x'=(\sqrt[m]{\varepsilon},0,\cdots,0)\}\cap\Omega,
\end{align}
where $\rho_{k+1;2}(n,m;\varepsilon)=\rho_{k+1}(n,m;\varepsilon)/\rho_{2}(n,m;\varepsilon)$.

\end{theorem}

\begin{remark}
Note that the main singularity of $|\nabla u|$ arises from $|\nabla u_{n+1}|$ under $(\mathbf{\widetilde{\Phi}1})$ or $(\mathbf{\widetilde{\Phi}2})$, while it comes from $|\nabla u_{0}|$ under $(\mathbf{\widetilde{\Phi}3})$.

\end{remark}

Before proving Theorem \ref{AAA666666}, we first estimate $|Q_{\beta}[\varphi]|$, $\beta=1,2,\cdots,\frac{n(n+1)}{2}$.

\begin{lemma}\label{OOT1}
Assume as in Theorem \ref{AAA666666}. Then,

$(a)$ if $(\mathbf{\widetilde{\Phi}1})$ holds, then
\begin{align}\label{LDAE1}
|Q_{\beta}[\varphi]|\leq&C(\eta\rho_{k}(n,m;\varepsilon)+\|\varphi\|_{C^{2}(\partial D)}),\quad\;\,\beta=1,2,\cdots,n,
\end{align}
and
\begin{align}\label{CCAAMM1}
|Q_{\beta}[\varphi]|\leq C(\eta+\|\varphi\|_{C^{2}(\partial D)}),\quad\;\,\beta=n+1,\cdots,\frac{n(n+1)}{2};
\end{align}

$(b)$ if $(\mathbf{\widetilde{\Phi}2})$ holds, then
\begin{align}\label{CCAAMM2}
|Q_{\beta}[\varphi]|\leq C(\eta+\|\varphi\|_{C^{2}(\partial D)}),\quad\;\,\beta=1,2,\cdots,n,
\end{align}
and
\begin{align}\label{LKM6}
\frac{\eta\rho_{k+1}(n,m;\varepsilon)}{C}\leq|Q_{2n-1}[\varphi]|\leq C\eta\rho_{k+1}(n,m;\varepsilon);
\end{align}

$(c)$ if $(\mathbf{\widetilde{\Phi}3})$ holds, then
\begin{align}\label{LADE2}
|Q_{\beta}[\varphi]|\leq C(\eta+\|\varphi\|_{C^{2}(\partial D)}),\quad\;\,\beta=1,2,\cdots,\frac{n(n+1)}{2}.
\end{align}
\end{lemma}
\begin{proof}
{\bf STEP 1.} Proofs of (\ref{LDAE1}), (\ref{CCAAMM1}), (\ref{CCAAMM2}) and (\ref{LADE2}). The proofs of (\ref{LDAE1}), (\ref{CCAAMM1}), (\ref{CCAAMM2}) and (\ref{LADE2}) are contained in the proof of Proposition \ref{lekm323} and thus omitted.

{\bf STEP 2.} Proof of (\ref{LKM6}). Denote
\begin{align*}
Q_{2n-1}[\varphi]=\sum^{n}_{j=1}\int_{\partial D_{1}}\frac{\partial u_{0j}}{\partial\nu_{0}}\Big|\cdot\psi_{2n-1},
\end{align*}
where $u_{0j}$, $j=1,2,\cdots,n$, are defined by (\ref{RTP101}). By definition, we split $Q_{2n-1}[\varphi]$ into two parts in the following.
\begin{align*}
Q_{2n-1}^{1}[\varphi]=&\sum^{n}_{j=1}\int_{\partial D_{1}}\left[\lambda\sum^{n}_{k=1}\partial_{k}u_{0j}^{k}\nu_{1}+\mu\sum^{n}_{i=1}(\partial_{i}u_{0j}^{1}+\partial_{1}u_{0j}^{i})\nu_{i}\right]x_{n}\\
&-\sum^{n}_{j=1}\int_{\partial D_{1}}\left[\lambda\sum^{n-1}_{k=1}\partial_{k}u_{0j}^{k}\nu_{n}+\mu\sum^{n-1}_{i=1}(\partial_{i}u_{0j}^{n}+\partial_{n}u_{0j}^{i})\nu_{i}\right]x_{1}\\
&-\sum^{n-1}_{j=1}(\lambda+2\mu)\int_{\partial D_{1}}\partial_{n}u_{0j}^{n}\nu_{n}x_{1},
\end{align*}
and
\begin{align*}
Q_{2n-1}^{2}[\varphi]=-(\lambda+2\mu)\int_{\partial D_{1}}\partial_{n}u_{0n}^{n}\nu_{n}x_{1}.
\end{align*}
In view of the fact that $|x_{n}|=|\varepsilon+h_{1}(x_{1})|\leq C(\varepsilon+|x_{1}|^{m})$ and $|\nu_{i}|\leq C|x'|^{m-1}$ on $\Gamma^{+}_{R}$, $i=1,\cdots,n-1$, it follows from $\tilde{u}^{n}_{0j}$=0, $j\neq n$, and Lemma \ref{lem89} that
\begin{align}\label{LKM1}
|Q_{2n-1}^{1}[\varphi]|\leq C(\eta+\|\varphi\|_{C^{2}(\partial D)}).
\end{align}

As for $Q_{2n-1}^{2}[\varphi]$, in view of (\ref{KKT66}), on one hand, we have
\begin{align}
|Q_{2n-1}^{2}[\varphi]|\leq&|\lambda+2\mu|\left|\int_{|x'|<R}\frac{\varphi^{n}(x',h(x'))x_{1}}{\varepsilon+h_{1}(x')-h(x')}\right|+C\|\varphi\|_{C^{2}(\partial D)}\notag\\
=&|\lambda+2\mu|\int_{|x'|<R}\frac{\eta|x_{1}|^{k+1}}{\varepsilon+|x'|^{m}}+C\|\varphi\|_{C^{2}(\partial D)}\notag\\
\leq&C(\eta\rho_{k+1}(n,m;\varepsilon)+\|\varphi\|_{C^{2}(\partial D)}).\label{LKM2}
\end{align}
On the other hand, we have
\begin{align}
|Q_{2n-1}^{2}[\varphi]|\geq&|\lambda+2\mu|\left|\int_{|x'|<R}\frac{\varphi^{n}(x',h(x')x_{1}}{\varepsilon+h_{1}(x')-h(x')}\right|-C\|\varphi\|_{C^{2}(\partial D)}\notag\\
=&|\lambda+2\mu|\int_{|x'|<R}\frac{\eta|x_{1}|^{k+1}}{\varepsilon+|x'|^{m}}-C\|\varphi\|_{C^{2}(\partial D)}\notag\\
\geq&\frac{\eta\rho_{k+1}(n,m;\varepsilon)}{C}-C\|\varphi\|_{C^{2}(\partial D)}.\label{LKM3}
\end{align}
Then combining with (\ref{LKM1})--(\ref{LKM3}), we obtain that (\ref{LKM6}) holds.

\end{proof}

\begin{lemma}\label{LKL}
Assume as in Theorem \ref{AAA666666}. Then, if $(\mathbf{\widetilde{\Phi}1})$ holds,
\begin{align*}
|Q_{2n-1}[\varphi]-Q^{\ast}_{2n-1}[\varphi]|\leq C\varepsilon^{\frac{n+k-m}{m(n+k)}}.
\end{align*}
Consequently,
\begin{align*}
Q_{2n-1}[\varphi]\rightarrow Q^{\ast}_{2n-1}[\varphi],\;\;\, \mathrm{as}\;\, \varepsilon\rightarrow0.
\end{align*}

\end{lemma}
\begin{proof}
Similarly as in Lemma \ref{lemmaLw.01}, we have
\begin{align*}
Q_{2n-1}[\varphi]-Q^{\ast}_{2n-1}[\varphi]=\int_{\partial D}\frac{\partial(u_{2n-1}-u_{2n-1}^{\ast})}{\partial\nu_{0}}\Big|_{+}\cdot\varphi(x).
\end{align*}
where $u_{n+1}^{\ast}$ solves
\begin{equation}\label{LKTl03.001}
\begin{cases}
\mathcal{L}_{\lambda,\mu}u_{2n-1}^{\ast}=0,\quad\;\,&\mathrm{in}\;\,\Omega^{\ast},\\
u_{2n-1}^{\ast}=\psi_{2n-1},\quad\;\,&\mathrm{on}\;\,\partial D_{1}^{\ast}\setminus\{0\},\\
u_{2n-1}^{\ast}=0,\quad\;\,&\mathrm{on}\;\,\partial D.
\end{cases}
\end{equation}

Similarly, we introduce two auxiliary functions
$$\tilde{u}_{2n-1}=(x_{n}\bar{v},0,\cdots,0,-x_{1}\bar{v}),\quad\mathrm{and}\quad\tilde{u}_{2n-1}^{\ast}=(x_{n}\bar{v}^{\ast},0,\cdots,0,-x_{1}\bar{v}^{\ast}),$$
where $\bar{v}$ and $\bar{v}^{\ast}$ are defined as before. It follows from ({\bf{\em H3}}) that for $x\in\Omega_{R}^{\ast}$,
\begin{align}\label{LKT6.003}
|\nabla_{x'}(\tilde{u}_{2n-1}-\tilde{u}^{\ast}_{2n-1})|\leq C,
\end{align}
and
\begin{align}\label{LKT6.004}
|\partial_{n}(\tilde{u}_{2n-1}-\tilde{u}^{\ast}_{2n-1})|\leq\frac{C\varepsilon}{|x'|^{m-1}(\varepsilon+|x'|^{m})}.
\end{align}
Applying Corollary \ref{co.001} to (\ref{LKTl03.001}), we derive that for $x\in\Omega_{R}^{\ast}$,
\begin{align}\label{LKT6.005}
|\nabla(u_{2n-1}^{\ast}-\tilde{u}_{2n-1}^{\ast})|\leq C,
\end{align}
and
\begin{align}\label{LKT6.006}
|\nabla_{x'}u_{2n-1}^{\ast}|\leq C,\quad\;|\partial_{n}u_{2n-1}^{\ast}|\leq\frac{C}{|x'|^{m-1}}.
\end{align}

We now divide into two steps to estimate $|Q_{2n-1}[\varphi]-Q^{\ast}_{2n-1}[\varphi]|$.

{\bf STEP 1.} Note that $u_{n+1}-u_{2n-1}^{\ast}$ satisfies
\begin{align*}
\begin{cases}
\mathcal{L}_{\lambda,\mu}(u_{2n-1}-u_{2n-1}^{\ast})=0,&\mathrm{in}\;\,D\setminus(\overline{D_{1}\cup D_{1}^{\ast}}),\\
u_{2n-1}-u_{2n-1}^{\ast}=\psi_{2n-1}-u_{2n-1}^{\ast},&\mathrm{on}\;\,\partial D_{1}\setminus D_{1}^{\ast},\\
u_{2n-1}-u_{2n-1}^{\ast}=u_{2n-1}-\psi_{2n-1},&\mathrm{on}\;\,\partial D_{1}^{\ast}\setminus(D_{1}\cup\{0\}),\\
u_{2n-1}-u_{2n-1}^{\ast}=0,&\mathrm{on}\;\,\partial D.
\end{cases}
\end{align*}
We first estimate $|u_{2n-1}-u_{2n-1}^{\ast}|$ on $\partial(D_{1}\cup D_{1}^{\ast})\setminus\mathcal{C}_{\varepsilon^{\gamma}}$, where $\mathcal{C}_{\varepsilon^{\gamma}}$ is defined by (\ref{LNM656}) and $0<\gamma<1/2$ to be determined later.

Similarly as before, for $x\in\partial D_{1}\setminus D_{1}^{\ast}$,
\begin{align}\label{LKT6.007}
|u_{2n-1}-u_{2n-1}^{\ast}|\leq C\varepsilon.
\end{align}
Utilizing (\ref{Le2.0291}), it follows that for $x\in\partial D_{1}^{\ast}\setminus(D_{1}\cup\mathcal{C}_{\varepsilon^{\gamma}})$,
\begin{align}\label{LKT6.008}
|u_{2n-1}-u_{2n-1}^{\ast}|\leq C\varepsilon^{1-(m-1)\gamma}.
\end{align}
Similarly, combining with (\ref{Lea2.0028}), (\ref{LKT6.004}) and (\ref{LKT6.005}), we deduce that for $x\in\Omega_{R}^{\ast}\cap\{|x'|=\varepsilon^{\gamma}\}$,
\begin{align*}
|\partial_{n}(u_{2n-1}-u_{2n-1}^{\ast})|\leq&\frac{C}{\varepsilon^{(2m-1)\gamma-1}}.
\end{align*}
Thus, due to the fact that $\tilde{u}_{2n-1}-\tilde{u}_{2n-1}^{\ast}=0$ on $\partial D$, we have
\begin{align}
|(u_{2n-1}-u_{2n-1}^{\ast})(x',x_{n})|\leq C\varepsilon^{1-(m-1)\gamma},\quad\;\,\mathrm{for}\;\,x\in\Omega_{R}^{\ast}\cap\{|x'|=\varepsilon^{\gamma}\}.\label{LKT6.009}
\end{align}
Choose $\gamma=\frac{1}{m}$. Then it follows from (\ref{LKT6.007}), (\ref{LKT6.008}) and (\ref{LKT6.009}) that
$$|u_{2n-1}-u_{2n-1}^{\ast}|\leq C\varepsilon^{\frac{1}{m}},\quad\;\,\mathrm{on}\;\,\partial\big(D\setminus\big(\overline{D_{1}\cup D_{1}^{\ast}\cup\mathcal{C}_{\varepsilon^{\frac{1}{m}}}}\big)\big).$$
In view of the maximum principle for Lam\'{e} system in \cite{MMN2007}, we have
$$|u_{2n-1}-u_{2n-1}^{\ast}|\leq C\varepsilon^{\frac{1}{m}},\quad\;\,\mathrm{on}\;\,D\setminus\big(\overline{D_{1}\cup D_{1}^{\ast}\cup\mathcal{C}_{\varepsilon^{\frac{1}{m}}}}\big).$$
Then using the standard interior and boundary estimates for Lam\'{e} system, we deduce that, for any $\frac{m-1}{m^{2}}<\tilde{\gamma}<\frac{1}{m}$,
$$|\nabla(u_{2n-1}-u_{2n-1}^{\ast})|\leq C\varepsilon^{m\tilde{\gamma}-\frac{m-1}{m}},\quad\;\,\mathrm{on}\;\,D\setminus\big(\overline{D_{1}\cup D_{1}^{\ast}\cup\mathcal{C}_{\varepsilon^{\frac{1}{m}-\tilde{\gamma}}}}\big),$$
which implies that
\begin{align}\label{LKT6.010}
|\mathcal{B}^{out}|:=\left|\int_{\partial D\setminus\mathcal{C}_{\varepsilon^{\frac{1}{m}-\tilde{\gamma}}}}\frac{\partial(u_{2n-1}-u_{2n-1}^{\ast})}{\partial\nu_{0}}\Big|_{+}\cdot\varphi(x)\right|\leq C\varepsilon^{m\tilde{\gamma}-\frac{m-1}{m}},
\end{align}
where $\frac{m-1}{m^{2}}<\tilde{\gamma}<\frac{1}{m}$ will be determined later.

{\bf STEP 2.} we next estimate
\begin{align*}
\mathcal{B}^{in}:=&\int_{\partial D\cap\mathcal{C}_{\varepsilon^{\frac{1}{m}-\tilde{\gamma}}}}\frac{\partial(u_{2n-1}-u_{2n-1}^{\ast})}{\partial\nu_{0}}\Big|_{+}\cdot\varphi(x)\\
=&\int_{\partial D\cap\mathcal{C}_{\varepsilon^{\frac{1}{m}-\tilde{\gamma}}}}\frac{\partial(\tilde{u}_{2n-1}-\tilde{u}_{2n-1}^{\ast})}{\partial\nu_{0}}\Big|_{+}\cdot\varphi(x)\\
&+\int_{\partial D\cap\mathcal{C}_{\varepsilon^{\frac{1}{m}-\tilde{\gamma}}}}\frac{\partial(w_{2n-1}-\tilde{w}_{2n-1}^{\ast})}{\partial\nu_{0}}\Big|_{+}\cdot\varphi(x)
=:\mathcal{B}_{\tilde{u}}+\mathcal{B}_{\tilde{w}},
\end{align*}
where $w_{2n-1}=u_{2n-1}-\tilde{u}_{2n-1}$, $w_{2n-1}^{\ast}=u_{2n-1}^{\ast}-\tilde{u}_{2n-1}^{\ast}$. First, we split $\mathcal{B}_{\tilde{u}}$ into two parts as follows.
\begin{align*}
\mathcal{B}^{1}_{\tilde{u}}=&\int_{\partial D\cap\mathcal{C}_{\varepsilon^{\frac{1}{m}-\tilde{\gamma}}}}\bigg\{\lambda\sum^{n}_{k=1}\partial_{1}(\tilde{u}^{1}_{2n-1}-\tilde{u}^{\ast1}_{2n-1})\nu_{k}\varphi^{k}+\lambda\sum^{n-1}_{k=1}\partial_{n}(\tilde{u}^{n}_{2n-1}-\tilde{u}^{\ast n}_{2n-1})\nu_{k}\varphi^{k}\\
&+\mu\sum^{n}_{k=1}\partial_{k}(\tilde{u}^{1}_{2n-1}-\tilde{u}^{\ast1}_{2n-1})[\nu_{k}\varphi^{1}+\nu_{1}\varphi^{k}]\\
&+\mu\sum^{n-1}_{k=1}\partial_{k}(\tilde{u}^{n}_{2n-1}-\tilde{u}^{\ast n}_{2n-1})[\nu_{n}\varphi^{k}+\nu_{k}\varphi^{n}]\bigg\},
\end{align*}
and
\begin{align*}
\mathcal{B}^{2}_{\tilde{u}}=&\int_{\partial D\cap\mathcal{C}_{\varepsilon^{\frac{1}{m}-\tilde{\gamma}}}}(\lambda+2\mu)\partial_{n}(\tilde{u}^{n}_{2n-1}-\tilde{u}^{\ast n}_{2n-1})\nu_{n}\varphi^{n}.
\end{align*}

Making use of (\ref{zw11}) and (\ref{LKT6.003})--(\ref{LKT6.004}), we obtain
\begin{align}\label{TYU612}
|\mathcal{B}^{1}_{\tilde{u}}|\leq&\int_{\partial D\cap\mathcal{C}_{\varepsilon^{\frac{1}{m}-\tilde{\gamma}}}}C|x'|^{k}\leq C\varepsilon^{(\frac{1}{m}-\tilde{\gamma})(n+k-1)},
\end{align}
and
\begin{align}\label{TYU613}
|\mathcal{B}^{2}_{\tilde{u}}|\leq&\int_{\partial D\cap\mathcal{C}_{\varepsilon^{\frac{1}{m}-\tilde{\gamma}}}}C|x'|^{k+1-m}\leq C\varepsilon^{(\frac{1}{m}-\tilde{\gamma})(n+k-m)}.
\end{align}
Then it follows from (\ref{TYU612}) and (\ref{TYU613}) that
\begin{align}\label{TYU615}
|\mathcal{B}_{\tilde{u}}|\leq&C\varepsilon^{(\frac{1}{m}-\tilde{\gamma})(n+k-m)}.
\end{align}

We now estimate $\mathcal{B}_{\tilde{w}}$. Applying Corollary \ref{co.001}, we derive that for $0<|x'|\leq R$,
\begin{align}\label{TYU616}
|\nabla w_{2n-1}(x)|\leq C,\;\,\,|\nabla w_{2n-1}^{\ast}(x)|\leq C.
\end{align}
By definition,
\begin{align*}
\mathcal{B}_{\tilde{w}}=&\int_{\partial D\cap\mathcal{C}_{\varepsilon^{\frac{1}{m}-\tilde{\gamma}}}}\bigg\{\lambda\sum^{n}_{i,j=1}\partial_{i}(w_{2n-1}^{i}-w_{2n-1}^{\ast i})\nu_{j}\varphi^{j}(x)\\
&+\mu\sum^{n}_{i,j=1}[\partial_{j}(w_{2n-1}^{i}-w_{2n-1}^{\ast i})+\partial_{i}(w_{2n-1}^{j}-w_{2n-1}^{\ast j})]\nu_{j}\varphi^{i}(x)\bigg\}.
\end{align*}
Then utilizing (\ref{zw11}) and (\ref{TYU616}), we derive
\begin{align*}
|\mathcal{B}_{\tilde{w}}|\leq C\varepsilon^{(\frac{1}{m}-\tilde{\gamma})(n+k-1)}.
\end{align*}
This, together with (\ref{TYU615}), yields that
\begin{align}\label{TYUU618}
|\mathcal{B}_{in}|\leq C\varepsilon^{(\frac{1}{m}-\tilde{\gamma})(n+k-m)}.
\end{align}
Thus, in light of (\ref{LKT6.010}), (\ref{TYUU618}) and picking $\tilde{\gamma}=\frac{n+k-1}{m(n+k)}$, we obtain
\begin{align*}
|Q_{2n-1}[\varphi]-Q^{\ast}_{2n-1}[\varphi]|\leq C\varepsilon^{\frac{n+k-m}{m(n+k)}}.
\end{align*}

\end{proof}

\noindent{\bf{Proof of Theorem \ref{AAA666666}.}} {\bf STEP 1.} Proof of (\ref{PPN2}).  Similarly as (\ref{Lw3.018}), it follows from assumptions $(\mathbf{\widetilde{\Phi}1})$ and $(\mathbf{\widetilde{\Phi}2})$ in Theorem \ref{AAA666666}, Lemma \ref{OOT1} and Lemma \ref{LKL} that there exists a universal constant $C_{0}>0$ such that for a sufficiently small $\varepsilon>0$,
\begin{align}\label{LPN6.018}
|Q_{2n-1}[\varphi]|>\frac{\rho_{k+1}(n,m;\varepsilon)}{C_{0}},
\end{align}
Denote $D^{\ast}=(a^{\ast}_{\alpha\beta})_{\frac{n(n-1)}{2}\times\frac{n(n-1)}{2}}$, where $a^{\ast}_{\alpha\beta}$ is the cofactor of $a_{\alpha\beta}$, $\alpha,\beta=n+1,\cdots,\frac{n(n+1)}{2}$. Then it follows from Lemma \ref{lemmabc} that for $n\geq3$,
\begin{align}
a^{\ast}_{\alpha\alpha}\thicksim&(\rho_{2}(n,m;\varepsilon))^{\frac{n(n-1)}{2}-1},\quad\;\,\alpha=n+1,\cdots,\frac{n(n+1)}{2},\label{LKT6.019}\\
a^{\ast}_{\alpha\beta}\thicksim&(\rho_{2}(n,m;\varepsilon))^{\frac{n(n-1)}{2}-2},\quad\;\,\alpha\neq\beta,\;\,\alpha,\beta=n+1,\cdots,\frac{n(n+1)}{2}.\label{LKT6.0191}
\end{align}

Recalling the definition of $C^{2n-1}$, it follows from Lemma \ref{lemmabc}, Lemma \ref{OOT1}, (\ref{GMD2}) and (\ref{LPN6.018})--(\ref{LKT6.0191}) that for a sufficiently small $\varepsilon>0$,
\begin{align}\label{LKT6.022}
|C^{2n-1}|\geq\frac{1}{2\det D}a^{\ast}_{2n-1,2n-1}|Q_{2n-1}[\varphi]|\geq\frac{\rho_{k+1}(n,m;\varepsilon)}{C\rho_{2}(n,m;\varepsilon)}.
\end{align}

In view of the definition of $\tilde{u}_{2n-1}$, we obtain that $\tilde{u}_{2n-1}^{n}=-x_{1}\bar{v}$. Thus, by Corollary \ref{co.001}, we deduce that for $x'=(\sqrt[m]{\varepsilon},0,\cdots,0)$,
\begin{align}\label{KLMN1}
|\partial_{n}u_{2n-1}^{n}|\geq&|\partial_{n}\tilde{u}^{n}_{2n-1}|-|\partial_{n}(u_{2n-1}^{n}-\tilde{u}^{n}_{2n-1})|\geq\frac{x_{1}}{C(\varepsilon+|x'|^{m})}\geq\frac{1}{C\varepsilon^{1-\frac{1}{m}}},
\end{align}
and for $n+1\leq\alpha\leq\frac{n(n+1)}{2}$, $\alpha\neq2n-1$,
\begin{align}\label{KLMN2}
|\partial_{n}u_{\alpha}^{n}|\leq&|\partial_{n}\tilde{u}^{n}_{\alpha}|+|\partial_{n}(u_{\alpha}^{n}-\tilde{u}^{n}_{\alpha})|\leq C,\quad\;\,\mathrm{for}\;x'=(\sqrt[m]{\varepsilon},0,\cdots,0).
\end{align}

It follows from (\ref{LKT6.022})--(\ref{KLMN2}) that for $x'=(\sqrt[m]{\varepsilon},0,\cdots,0)$,
\begin{align}\label{KLMN3}
\Big|\sum^{\frac{n(n+1)}{2}}_{\alpha=n+1}C^{\alpha}\nabla u_{\alpha}\Big|\geq\Big|\sum^{\frac{n(n+1)}{2}}_{\alpha=n+1}C^{\alpha}\partial_{n}u_{\alpha}^{n}\Big|
\geq&\left|C^{2n-1}\partial_{n}u_{2n-1}^{n}\right|-\Big|\sum^{\frac{n(n+1)}{2}}_{\alpha=n+1,\alpha\neq2n-1}C^{\alpha}\partial_{n}u_{\alpha}^{n}\Big|\notag\\
\geq&\frac{\rho_{k+1}(n,m;\varepsilon)}{C\varepsilon^{1-\frac{1}{m}}\rho_{2}(n,m;\varepsilon)},
\end{align}
while, by means of Corollary \ref{co.001}, Lemma \ref{lem89}, Proposition \ref{proposition123} and Proposition \ref{lekm323}, we obtain that
\begin{align}\label{KLMN6}
|\nabla u_{0}|\leq&\frac{C}{\varepsilon^{1-\frac{k}{m}}},
\end{align}
and

$(i)$ if $(\mathbf{\widetilde{\Phi}1})$ holds,
\begin{align}\label{KLMN5}
\left|\sum^{n}_{\alpha=1}C^{\alpha}\nabla u_{\alpha}\right|\leq&
\frac{C\rho_{k}(n,m;\varepsilon)}{\varepsilon^{\frac{n-1}{m}}},&k>m-n,\;m>n;
\end{align}

$(ii)$ if $(\mathbf{\widetilde{\Phi}2})$ holds,
\begin{align}\label{KLMN9}
\left|\sum^{n}_{\alpha=1}C^{\alpha}\nabla u_{\alpha}\right|\leq&
\frac{C}{\varepsilon^{\frac{n-1}{m}}},& k=m-n,\;m>n+1.
\end{align}
Then it follows from (\ref{Le2.015}) and (\ref{KLMN3})--(\ref{KLMN9}) that for $x'=(\sqrt[m]{\varepsilon},0,\cdots,0)$,
\begin{align*}
|\nabla u(x)|\geq\frac{\rho_{k+1}(n,m;\varepsilon)}{C\varepsilon^{1-\frac{1}{m}}\rho_{2}(n,m;\varepsilon)}.
\end{align*}
That is, we prove (\ref{PPN2}).

{\bf STEP 2.} Proof of (\ref{PPNM6}). It follows from Lemma \ref{lem89} that for $x'=(\sqrt[m]{\varepsilon},0,\cdots,0)$,
\begin{align}\label{CRKN1}
|\nabla u_{0}(x)|\geq\frac{|\varphi^{1}(x',h(x'))|}{C(\varepsilon+|x'|^{m})}\geq\frac{|x_{1}|^{k}}{C(\varepsilon+|x_{1}|^{m})}\geq\frac{1}{C\varepsilon^{1-\frac{k}{m}}},
\end{align}
while, in light of Corollary \ref{co.001}, Proposition \ref{proposition123} and Lemma \ref{OOT1}, we deduce that if condition $(\mathbf{\widetilde{\Phi}3})$ in Theorem \ref{AAA666666} holds,
\begin{align*}
\left|\sum^{n}_{\alpha=1}C^{\alpha}\nabla u_{\alpha}(x)\right|\leq\frac{C}{\varepsilon^{\frac{n-1}{m}}},\quad\mbox{and}\quad
\left|\sum^{\frac{n(n+1)}{2}}_{\alpha=n+1}C^{\alpha}\nabla u_{\alpha}\right|\leq\frac{C}{\varepsilon\rho_{2}(n,m;\varepsilon)}.
\end{align*}
This, combining with (\ref{CRKN1}), completes the proof of (\ref{PPNM6}).

%

\begin{remark}
If conditions (\ref{KHMDC1}) and (\ref{KKT66}) hold in Theorem \ref{AAA666666} for $k=1$, $m=n+1$ or for $k<m-n$,  we don't give the lower bound due to the fact that terms $\frac{|x'|}{\varepsilon+|x'|^{m}}\frac{\rho_{k+1}(n,m;\varepsilon)}{\rho_{2}(n,m;\varepsilon)}$ and $\frac{|x'|^{k}}{\varepsilon+|x'|^{m}}$ simultaneously attain the greatest blow-up rate at $|x'|=\sqrt[m]{\varepsilon}$. Similarly, if conditions (\ref{KHMDC1}) and (\ref{PDSQ131}) hold in Theorem \ref{AAA666666} for $k=m-n$, $m>n+1$, there also appears the same phenomenon as before for terms $\frac{|x'|}{\varepsilon+|x'|^{m}}\frac{1}{\rho_{2}(n,m;\varepsilon)}$ and $\frac{|x'|^{k}}{\varepsilon+|x'|^{m}}$.
\end{remark}

\section{Appendix:\,The proof of Theorem \ref{thm86}}

Recalling assumptions ({\bf{\em H1}}) and ({\bf{\em H2}}), we obtain
\begin{align*}
\frac{1}{C}(\varepsilon+d^{m}(x'))\leq \delta(x')\leq C(\varepsilon+d^{m}(x')).
\end{align*}
A direct calculation gives that for $i=1,2,\cdots,n$, $j=1,\cdots,n-1$, $x\in\Omega_{2R}$,
\begin{align}
|\partial_{j}\tilde{v}_{i}|\leq&\frac{C|\psi^{i}(x',\varepsilon+h_{1}(x'))|}{\sqrt[m]{\varepsilon+d^{m}(x')}}+C\|\nabla\psi^{i}\|_{L^{\infty}(\partial D_{1})},\label{QQ.101}\\
|\partial_{n}\tilde{v}_{i}|=&\frac{|\psi^{i}(x',\varepsilon+h_{1}(x'))|}{\delta(x')},\quad\partial_{nn}\tilde{v}_{i}=0,\label{QQ.102}
\end{align}
and
\begin{align}
|\mathcal{L}_{\lambda,\mu}\tilde{v}_{i}|\leq&C|\nabla^2\tilde{v}_{i}|\notag\\
\leq&\frac{C|\psi^{i}(x',\varepsilon+h_{1}(x'))|}{(\varepsilon+d^{m}(x'))^{1+\frac{1}{m}}}+\frac{C\|\nabla\psi^{i}\|_{L^{\infty}(\partial D_{1})}}{\varepsilon+d^{m}(x')}+C\|\nabla^{2}\psi^{i}\|_{L^{\infty}(\partial D_{1})}.\label{QQ.103}
\end{align}
Here and throughout this section, in order to simplify notations, we use $\|\nabla\psi^{i}\|_{L^{\infty}}$, $\|\nabla^{2}\psi^{i}\|_{L^{\infty}}$ and $\|\psi\|_{C^{2}}$ to denote $\|\nabla\psi^{i}\|_{L^{\infty}(\partial{D}_{1})}$, $\|\nabla^{2}\psi^{i}\|_{L^{\infty}(\partial{D_1})}$ and $\|\psi^{i}\|_{C^{2}(\partial D_{1})}$, respectively.

Define
\begin{equation}\label{def_w}
w_i:=v_i-\tilde{v}_i,\qquad i=1,2,\cdots,n.
\end{equation}

\noindent{\bf STEP 1.}
Let $v_i\in H^1(\Omega; \mathbb{R}^{n})$ be a weak solution of (\ref{P2.010}). Then
\begin{align}\label{lem2.2equ}
\int_{\Omega}|\nabla w_i|^2dx\leq C\|\psi^{i}\|_{C^{2}(\partial D_{1})}^{2},\qquad\,i=1,2,\cdots,n.
\end{align}

For simplicity, we denote
$$w:=w_{i},\quad\mbox{and}\quad \tilde{v}:=\tilde{v}_{i}.$$
Then, $w$ satisfies
\begin{align}\label{eq2.6}
\begin{cases}
  \mathcal{L}_{\lambda,\mu}w=-\mathcal{L}_{\lambda,\mu}\tilde{v},&
\hbox{in}\  \Omega,  \\
w=0, \quad&\hbox{on} \ \partial\Omega.
\end{cases}
\end{align}
Multiplying the equation in (\ref{eq2.6}) by $w$ and integrating by parts, we have
\begin{align}\label{integrationbypart}
\int_{\Omega}\left(\mathbb{C}^0e(w),e(w)\right)dx=\int_{\Omega}w\left(\mathcal{L}_{\lambda,\mu}\tilde{v}\right)dx.
\end{align}

Using the Poincar\'{e} inequality, we have
\begin{align}\label{poincare_inequality}
\|w\|_{L^2(\Omega\setminus\Omega_R)}\leq\,C\|\nabla w\|_{L^2(\Omega\setminus\Omega_R)},
\end{align}
while, in light of the Sobolev trace embedding theorem,
\begin{align}\label{trace}
\int\limits_{\scriptstyle |x'|={R},\atop\scriptstyle
h(x')<x_{n}<{\varepsilon}+h_1(x')\hfill}|w|\,dx\leq\ C \left(\int_{\Omega\setminus\Omega_{R}}|\nabla w|^2dx\right)^{\frac{1}{2}},
\end{align}
where the constant $C$ is independent of $\varepsilon$. Utilizing (\ref{QQ.101}), we get
\begin{align}
\int_{\Omega_{R}}|\nabla_{x'}\tilde{v}|^2dx\leq& C\int_{|x'|<R}\delta(x')\left(\frac{|\psi^{i}(x',\varepsilon+h_{1}(x'))|^{2}}{(\varepsilon+d^m(x'))^\frac{2}{m}}+\|\nabla\psi^{i}\|_{L^{\infty}}^{2}\right)dx'\notag\\
\leq&C\|\psi^{i}\|_{C^{1}(\partial{D}_{1})}^{2}.\label{nablax'tildeu}
\end{align}

Combining with \eqref{Le2.012}--\eqref{Le2.01222}, \eqref{poincare_inequality} and the first Korn's inequality, we deduce that
\begin{align}
\int_{\Omega}|\nabla w|^2dx\leq &\,2\int_{\Omega}|e(w)|^2dx\nonumber\\
\leq&\,C\left|\int_{\Omega_R}w(\mathcal{L}_{\lambda,\mu}\tilde{v})dx\right|+C\left|\int_{\Omega\setminus\Omega_R}w(\mathcal{L}_{\lambda,\mu}\tilde{v})dx\right|\nonumber\\
\leq&\,C\left|\int_{\Omega_R}w(\mathcal{L}_{\lambda,\mu}\tilde{v})dx\right|+C\|\psi^{i}\|_{C^2(\partial{D}_{1})}\int_{\Omega\setminus\Omega_R}|w|dx\nonumber\\
\leq&\,C\left|\int_{\Omega_R}w(\mathcal{L}_{\lambda,\mu}\tilde{v})dx\right|+C\|\psi^{i}\|_{C^2(\partial{D}_{1})}\|\nabla w\|_{L^{2}(\Omega\setminus\Omega_{R})},\nonumber
\end{align}
while, according to (\ref{QQ.102}), (\ref{trace}) and (\ref{nablax'tildeu}),
\begin{align}
&\left|\int_{\Omega_R}w(\mathcal{L}_{\lambda,\mu}\tilde{v})dx\right|
\leq\,C\sum_{k+j<2n}\left|\int_{\Omega_{R}}w\partial_{kj}\tilde{v}dx\right|\nonumber\\
\leq&C\int_{\Omega_{R}}|\nabla w\|\nabla_{x'}\tilde{v}|dx+\int\limits_{\scriptstyle |x'|={R},\atop\scriptstyle
h(x')<x_{n}<\varepsilon+h_1(x')\hfill}C|\nabla_{x'}\tilde{v}\|w|dx\nonumber \\
\leq&C\|\nabla w\|_{L^{2}(\Omega_{R})}\|\nabla_{x'}\tilde{v}\|_{L^{2}(\Omega_{R})}+C\|\psi^{i}\|_{C^1(\partial{D}_{1})}\|\nabla w\|_{L^{2}(\Omega\setminus\Omega_{R})}\nonumber\\
\leq&C\|\psi^{i}\|_{C^{1}(\partial{D}_{1})}\|\nabla w\|_{L^{2}(\Omega)}.\nonumber
\end{align}
Thus,
\begin{align*}
\|\nabla w\|_{L^{2}(\Omega)}\leq  C\|\psi^{i}\|_{C^{2}(\partial{D}_{1})}.
\end{align*}
\noindent{\bf STEP 2.}
Proof of
\begin{align}\label{step2}
 \int_{\Omega_\delta(z')}|\nabla w|^2dx
 &\leq C\delta^{n-\frac{2}{m}}\left(|\psi^{i}(z',\varepsilon+h_{1}(z'))|^2+\delta^{\frac{2}{m}}\|\psi^{i}\|_{C^2(\partial D_1)}^2\right).
\end{align}
As seen in \cite{BJL2017}, we have the iteration formula as follows:
\begin{align*}
\int_{\Omega_{t}(z')}|\nabla w|^{2}dx\leq\frac{C}{(s-t)^{2}}\int_{\Omega_{s}(z')}|w|^{2}dx+C(s-t)^{2}\int_{\Omega_{s}(z')}|\mathcal{L}_{\lambda,\mu}\tilde{v}|^{2}dx.
\end{align*}

We now divide into two cases to get (\ref{step2}).

{\bf Case 1.} If $z'\in\Sigma'_{\sqrt[m]{\varepsilon}}:=\{x'\in B'_{R}\,|\;\mathrm{dist}(x',\Sigma')<\sqrt[m]{\varepsilon}\},\;0<s<\sqrt[m]{\varepsilon},$ we get $\varepsilon\leq\delta(x')\leq C\varepsilon$ in $\Omega_{\sqrt[m]{\varepsilon}}(z')$. Using (\ref{QQ.103}), we obtain
\begin{align}\label{AQ3.037}
&\int_{\Omega_{s}(z')}|\mathcal{L}_{\lambda,\mu}\tilde{v}|^{2}\notag\\
&\leq C|\psi^{i}(z',\varepsilon+h_{1}(z'))|^{2}\frac{s^{n-1}}{\varepsilon^{\frac{m+2}{m}}}+C\|\nabla\psi^{i}\|^{2}_{L^{\infty}}\frac{s^{n-1}}{\varepsilon}+C\varepsilon s^{n-1}\|\nabla^{2}\psi^{i}\|^{2}_{L^{\infty}},
\end{align}
while, in view of the fact that $w=0$ on $\Gamma^{-}_{R}:=\{x\in\mathbb{R}^{n}|\,x_{n}=h(x'),\,|x'|<R\}$,
\begin{align}\label{AQ3.038}
\int_{\Omega_{s}(z')}|w|^{2}\leq C\varepsilon^{2}\int_{\Omega_{s}(z')}|\nabla w|^{2}.
\end{align}
Denote
$$F(t):=\int_{\Omega_{t}(z')}|\nabla w_{1}|^{2}.$$
It follows from (\ref{AQ3.037}) and (\ref{AQ3.038}) that for $0<t<s<\sqrt[m]{\varepsilon}$,
\begin{align}\label{AQ3.039}
F(t)\leq &\left(\frac{c_1\varepsilon}{s-t}\right)^2F(s)+C(s-t)^2s^{n-1}\bigg(\frac{|\psi^{i}(z',\varepsilon+h_{1}(z'))|^2}{\varepsilon^{1+\frac{2}{m}}}+\frac{\|\psi^{i}\|_{C^2}^2}{\varepsilon}\bigg),
\end{align}
where $c_{1}$ and $C$ are universal constants, independent of $|\Sigma'|$.

Take $k=\left[\frac{1}{4c_{1}\sqrt[m]{\varepsilon}}\right]+1$ and $t_{i}=\delta+2c_{1}i\varepsilon,\;i=0,1,2,\cdots,k.$ Then, (\ref{AQ3.039}) together with $s=t_{i+1}$ and $t=t_{i}$ leads to
$$F(t_{i})\leq\frac{1}{4}F(t_{i+1})+C(i+1)^{n-1}\varepsilon^{n-\frac{2}{m}}\left[|\psi^{i}(z',\varepsilon+h_{1}(z'))|^{2}+\varepsilon^{\frac{2}{m}}\|\psi^{i}\|^{2}_{C^{2}}\right].$$
It follows from $k$ iterations and (\ref{lem2.2equ}) that for a sufficiently small $\varepsilon>0$,
\begin{align}\label{AQ3.043}
F(t_{0})\leq C\varepsilon^{n-\frac{2}{m}}\left(|\psi^{i}(z',\varepsilon+h_{1}(z'))|^{2}+\varepsilon^{\frac{2}{m}}\|\psi^{i}\|^{2}_{C^{2}}\right).
\end{align}

{\bf Case 2.} If $z'\in B_{R}'\setminus\Sigma'_{\sqrt[m]{\varepsilon}}$ and $0<s<\frac{2d(z')}{3}$, we have $\frac{d^{m}(z')}{C}\leq\delta(x')\leq Cd^{m}(z')$ in $\Omega_{\frac{2d(z')}{3}}(z')$. Like (\ref{AQ3.037}) and (\ref{AQ3.038}), we have
\begin{align*}
\int_{\Omega_{s}(z')}|\mathcal{L}_{\lambda,\mu}\tilde{v}|^{2}\leq C|\psi^{i}(z',\varepsilon+h_{1}(z'))|^{2}\frac{s^{n-1}}{d^{m+2}(z')}+C\|\psi^{i}\|^{2}_{C^{2}}\frac{s^{n-1}}{d^{m}(z')},
\end{align*}
and
$$\int_{\Omega_{s}(z')}|w|^{2}\leq Cd^{2m}(z')\int_{\Omega_{s}(z')}|\nabla w|^{2}.$$
Further, for $0<t<s<\frac{2d(z')}{3}$, estimate (\ref{AQ3.039}) turns into
\begin{align*}
F(t)\leq\left(\frac{c_{2}d^{m}(z')}{s-t}\right)^{2}F(s)+C(s-t)^{2}s^{n-1}\left(\frac{|\psi^{i}(z',\varepsilon+h_{1}(z'))|^{2}}{d^{m+2}(z')}+\frac{\|\psi^{i}\|_{C^{2}}^{2}}{d^{m}(z')}\right),
\end{align*}
where $c_{2}$ and $C$ are universal constants, independent of $|\Sigma'|$.

Similarly as before, choose $k=\left[\frac{1}{4c_{2}d(z')}\right],\;t_{i}=\delta+2c_{2}id^{m}(z'),\,i=0,1,2,\cdots,k$ and set $s=t_{i+1},\;t=t_{i}$. Then, we have
$$F(t_{i})\leq\frac{1}{4}F(t_{i+1})+C(i+1)^{n-1}d^{mn-2}(z')\left(|\psi^{i}(z',\varepsilon+h_{1}(z'))|^{2}+d^{2}(z')\|\psi^{i}\|^{2}_{C^{2}}\right).$$
Likewise, via $k$ iterations, we obtain
\begin{align}\label{AQ3.046}
F(t_{0})\leq Cd^{mn-2}(z')\left(|\psi^{i}(z',\varepsilon+h_{1}(z'))|^{2}+d^{2}(z')\|\psi^{i}\|^{2}_{C^{2}}\right).
\end{align}

Thus, combining with (\ref{AQ3.043}) and (\ref{AQ3.046}), we obtain (\ref{step2}).

\noindent{\bf STEP 3.}
Proof of that for $i=1,2,\cdots,n$,
\begin{align}\label{AQ3.052}
|\nabla w_{i}(x)|\leq\frac{C|\psi^{i}(x',\varepsilon+h_{1}(x'))|}{\sqrt[m]{\varepsilon+d^{m}(x')}}+C\|\psi^{i}\|_{C^{2}(\partial D_{1})},\quad\mathrm{in}\;\,\Omega_{R}.
\end{align}

As in \cite{BLL2015} and\cite{BJL2017}, utilizing the rescaling argument, Sobolev embedding theorem, $W^{2,p}$ estimate and bootstrap argument, we derive
\begin{align*}
\|\nabla w\|_{L^{\infty}(\Omega_{\delta/2}(z'))}\leq\frac{C}{\delta}\left(\delta^{1-\frac{n}{2}}\|\nabla w\|_{L^{2}(\Omega_{\delta}(z'))}+\delta^{2}\|\mathcal{L}_{\lambda,\mu}\tilde{v}\|_{L^{\infty}(\Omega_{\delta}(z'))}\right).
\end{align*}

{\bf Case 1.} For $z'\in\Sigma'_{\sqrt[m]{\varepsilon}}$. Due to (\ref{QQ.103}), (\ref{step2}) and $\varepsilon\leq\delta(z')\leq C\varepsilon$, we obtain
$$\delta\|\mathcal{L}_{\lambda,\mu}\tilde{v}\|_{L^{\infty}(\Omega_{\delta}(z'))}\leq\frac{C|\psi^{i}(z',\varepsilon+h_{1}(z'))|}{\sqrt[m]{\delta}}+C\|\psi^{i}\|_{C^{2}},$$
and
\begin{align*}
\delta^{-\frac{n}{2}}\|\nabla w\|_{L^{2}(\Omega_{\delta}(z'))}\leq\frac{C|\psi^{i}(z',\varepsilon+h_{1}(z'))|}{\sqrt[m]{\delta}}+C\|\psi^{i}\|_{C^{2}}.
\end{align*}
Consequently, for $h(z')<z_{n}<\varepsilon+h_{1}(z')$,
$$|\nabla w(z',z_{n})|\leq\frac{C|\psi^{i}(z',\varepsilon+h_{1}(z'))|}{\sqrt[m]{\delta}}+C\|\psi^{i}\|_{C^{2}}.$$

{\bf Case 2.} For $z'\in B'_{R}\setminus\Sigma'_{\sqrt[m]{\varepsilon}}$. Similar to {\bf Case 1}, based on the fact that $\frac{d^{m}(z')}{C}\leq\delta(x')\leq Cd^{m}(z')$ in $\Omega_{\delta}(z')$, it follows from (\ref{QQ.103}) and (\ref{step2}) that for $h(z')<z_{n}<\varepsilon+h_{1}(z')$,
$$|\nabla w_{1}(z',z_{n})|\leq\frac{C|\psi^{i}(z',\varepsilon+h_{1}(z'))|}{\sqrt[m]{\delta}}+C\|\psi^{i}\|_{C^{2}}.$$

Therefore, estimate (\ref{AQ3.052}) is established. On the other hand, it follows from the standard interior estimates and boundary estimates for elliptic systems (\ref{P2.008}) (see Agmon et al. \cite{ADN1959} and \cite{ADN1964}) that
\begin{align*}
\|\nabla v\|_{L^{\infty}(\Omega\setminus\Omega_{R})}\leq C\|\psi\|_{C^{2}(\partial D_{1})}.
\end{align*}
Thus, Theorem \ref{thm86} is proved.

\noindent{\bf{\large Acknowledgements.}} H. Li was partially supported by NSFC (11571042, 11631002, 11971061) and BJNSF (1202013). Z. Zhao would like to thank Xia Hao for her helpful suggestions and discussions.

\bibliographystyle{plain}

\def\cprime{$'$}

\end{document}